\documentclass[10pt]{amsart}
\usepackage[a4paper,margin=1.2in,footskip=0.25in]{geometry}
\usepackage{ed}
\usepackage{makecell}

\usepackage{tablefootnote}
\usepackage{chngcntr}
\usepackage{lipsum, longtable}
\counterwithin{table}{equation}

\subjclass[2020]{Primary: 11G05, Secondary: 11G07, 11G40}

\begin{document}
\title{Tamagawa numbers and positive rank of elliptic curves}
\author{Edwina Aylward}

\address{University College London, London WC1H 0AY, UK}
\email{edwina.aylward.23@ucl.ac.uk}

\begin{abstract}
This paper addresses the prediction of positive rank for elliptic curves without the need to find a point of infinite order or compute L-functions. While the most common method relies on parity conjectures, a recent technique introduced by Dokchitser, Wiersema, and Evans predicts positive rank based on the value of a certain product of Tamagawa numbers, raising questions about its relationship to parity.  We show that their method is a subset
of the parity conjectures approach: whenever their method predicts positive rank, so does the
use of parity conjectures. To establish this, we extend previous work on Brauer relations and regulator constants to a broader setting involving combinations of permutation modules known as K-relations. A central ingredient in our argument is demonstrating a compatibility between Tamagawa numbers and local root numbers. 
\end{abstract}

\maketitle
\tableofcontents

\section{Introduction}

Determining the Mordell-Weil rank of an elliptic curve is a notoriously difficult task.  
As such, the parity conjecture serves as a powerful tool for predicting rank behaviour. Associated to an elliptic curve $E$ over a number field $L$ is its global root number $w(E / L) \in \{\pm1\}$.  
The parity conjecture states that this root number determines the parity of the rank of $E / L$: 
$$ w(E / L) = (-1)^{\rk E / L}. $$
Since root numbers can be computed explicitly from local arithmetic data (see \cite[\S2]{lbd}), they offer a practical method for analysing rank behaviour. The use of root number computations and parity conjectures to predict positive rank is often referred to as a {\em parity test}. 

For example, if $E / L$ has semistable reduction, then by \cite[Corollary 2.5]{lbd}, 
$$ w(E / L) = (-1)^{m + s},$$
where $m$ is the number of places of $L$ at which $E$ has split multiplicative reduction, and $s$ is the number of infinite places of $L$. Thus the parity conjecture predicts that
any semistable rational elliptic curve $E / \bQ$ with an even number of places of split multiplicative reduction must have positive (odd) rank. No other approach obtains this result as effortlessly.

In \cite{dok-wier-ev}, Dokchitser, Evans, and Wiersema introduce a novel method for deducing that a rational elliptic curve $E / \bQ$ has positive rank over a Galois extension $F / \bQ$, based on expected properties of $L$-functions. 
This is reminiscent of using root number computations to predict positive rank, as their method also involves local arithmetic data. 
To use this method, one computes a product of local data consisting of Tamagawa numbers and other local factors, and checks whether it is a norm from a specified quadratic field. If it is not, then $E$ must have positive rank over $F$.
For the purpose of comparison, we call using this method to test for positive rank a {\em norm relations test}.

A precise formulation of the norm relations test is given below in Conjecture \ref{nrt}. We first give an example of its use. 

\begin{example}\label{ex-d42}
	Consider a Galois extension $F / \bQ$ with $\Gal(F / \bQ) = D_{21}$ (the dihedral group of order $42$). Let $E / \bQ$ be a semistable elliptic curve. For $H \leq G$, let $c(E / F^{H}) = \prod_v c_v(E / F^{H})$ be the product of local Tamagawa numbers of $E / F^{H}$ at all finite places. According to the norm relations test, if
	$$ \frac{c(E / F^{C_2}) \cdot c(E / \bQ)}{c(E / F^{D_7}) \cdot c(E / F^{S_3})} $$
	is not the norm of an element of $\bQ(\sqrt{21})$, then $\rk E / F > 0$. 
	For instance, let $p$ be a prime with residue degree $2$ and ramification degree $3$ in $F$. If $E / \bQ$ is a semistable elliptic curve with split multiplicative reduction at $p$ and good reduction at all other ramified primes in $F$, then the product of Tamagawa numbers is not a norm and so the test predicts that $\rk E / F > 0$.  
\end{example}

A natural question concerns the precise relationship between the norm relations test and parity tests. One might ask, for instance, whether there exists an example in which the norm relations test predicts positive rank while root number computations fail to do so. In 
\cite[Remark 11]{dok-wier-ev}, the authors reported that they were unable to identify such an example. The principal result of this paper establishes that no such example can exist: whenever the norm relations test predicts positive rank, parity tests necessarily yield the same prediction. This confirms the empirical observation in \cite{dok-wier-ev} and demonstrates that the scope of the norm relations test does not extend beyond that of parity-based methods.


To make this precise, we require twisted root numbers. If $F / L$ is a Galois extension and $E /  L$ an elliptic curve, then $E(F)_{\bC} := E(F) \otimes_{\bZ} \bC$ is a complex representation of $G = \Gal(F / L)$. For a self-dual representation $\rho$ of $G$, there is an associated twisted root number $w(E / L, \rho) \in \{ \pm 1 \}$, whose construction is given in \cite[\S3 - \S4]{RohG}. Basic properties of twisted root numbers can be found in \cite[Appendix A]{tamroot}, and are also recalled in Section 4 of this paper (see Remark \ref{remark:root_numbers}). In particular, these satisfy multiplicativity: $w(E / L, \rho \oplus \rho') = w(E / L, \rho)w(E / L, \rho')$ for self-dual representations $\rho$, $\rho'$, and $w(E / L, \Ind_H^G \trivial) = w(E / F^{H})$ for any subgroup $H \leq G$. 
The parity conjecture for twists relates this root number to the multiplicity of $\rho$ as a sub-representation of $E(F)_{\bC}$, stating that 
$$ w(E / L, \rho) = (-1)^{\langle \rho, E(F)_{\bC} \rangle},$$
where $\langle \cdot , \cdot \rangle$ denotes the usual inner-product of representations. 
In particular, if $w(E / L, \rho) = -1$ then $\rho$ appears as a subrepresentation of $E(F)_{\bC}$ and so $\rk E / F > 0$. 

The relation we establish between the norm relations test and root number computations is as follows:

\begin{theorem}[{Theorem \ref{thm:weak_app}}]\label{thm:weak}
	Let $E / \bQ$ be an elliptic curve with semistable reduction at $2$ and $3$, $F / \bQ$ a Galois extension. Suppose that the norm relations test predicts that $\rk E / F > 0$. Then there exists an irreducible orthogonal representation $\rho$ of $\Gal(F / \bQ)$ such that $w(E / \bQ, \rho) = -1$. \end{theorem}

For example, consider an odd-order Galois extension $F / \bQ$, and an elliptic curve $E / \bQ$ satisfying the conditions of the theorem. Since the trivial representation is the only self-dual irreducible representation of $\Gal(F / \bQ)$, one has $w(E / \bQ) = w(E / F)$ (cf. \cite[Proposition A.2]{tamroot}). Thus parity tests cannot predict rank growth, and accordingly, the norm relations test cannot predict that $\rk E / F > 0$ (Corollary \ref{corr:fail}).

\subsection{Norm relations test}
The norm relations test, introduced in \cite[\S3]{dok-wier-ev}, is based on the conjectural properties of the {\em Birch--Swinnerton-Dyer quotient} associated to an elliptic curve $E$ over a number field $F$:
$$ \BSD(E / F) := \frac{\Reg_{E / F} |\sha_{E / F}| C_{E / F}}{|E(F)_{\tors}|^2},$$
where it is assumed that the Tate-Shafarevich group $\sha_{E / F}$ has finite order. This quotient appears as a factor in the conjectural formula for the leading coefficient of the Taylor expansion of the $L$-function $L(E / F, s)$ at $s = 1$ (the missing terms being the periods and $\sqrt{\Delta_F}$). 
Here, $C_{E / F}$ denotes the product of Tamagawa numbers and other local fudge factors at finite places of additive reduction (defined in \S\ref{sec:notation}). Given an elliptic curve, these are relatively straight-forward to compute, and are the terms one computes in the norm relations test. 

Let $F / \bQ$ be a Galois extension with $G = \Gal(F / \bQ)$, and $\rho$ an irreducible representation of $G$. Write $\bQ(\rho)$ for the field generated by the values of the character of $\rho$. For $\sigma \in \Gal(\bQ(\rho) / \bQ)$, let $\rho^{\sigma}$ be the representation of $G$ with $\Tr \rho^{\sigma} (g) = \sigma(\Tr \rho(g))$ for all $g \in G$. The norm relations test is as follows: 

\begin{conjecture}[{Norm relations test, \cite[Conjecture 4, Theorem 33]{dok-wier-ev}}]\label{nrt}
	Let $E$ be an elliptic curve over $\bQ$, $F / \bQ$ a finite Galois extension with $G = \Gal(F / \bQ)$, and $\rho$ an irreducible representation of $G$.  
Let
\begin{equation*}\label{eq:nrt}
	 \bigg{(} \bigoplus_{\sigma \in \Gal(\bQ(\rho) / \bQ)} \rho^{\sigma} \bigg{)}^{\oplus m}
	 = \bigg{(} \bigoplus_i \bC[G / H_i] \bigg{)} \ominus \bigg{(} \bigoplus_j \bC[G / H_j'] \bigg{)}
\end{equation*}
for some $m \geq 1$ and $H_i$, $H_j' \leq G$. If either
$ \prod_i C_{E / F^{H_i}} \big{/} \prod_j C_{E / F^{H_j'}} $
is not a norm from a quadratic field $\bQ(\sqrt{D}) \subset \bQ(\rho)$ (when such a subfield exists), or is not a rational square when $m$ is even, then $E$ has a point of infinite order over $F$. 
\end{conjecture}

By \cite[Theorem 5]{dok-wier-ev}, Conjecture \ref{nrt} is true  
assuming the analytic continuation of the $L$-functions $L(E, \tau, s)$ for every representation $\tau$ of $G$, their functional equation, the Birch--Swinnerton-Dyer conjecture, and Deligne's period conjecture. We show that it also follows from the parity conjecture for twists, removing the dependence of the result on $L$-functions. 

\begin{theorem}[{Corollary \ref{corr:nrt_true}}]
	Conjecture \ref{nrt} holds with $\bQ$ replaced by an arbitrary number field $L$ for all elliptic curves $E / L$ with semi-stable reduction at primes above $2$ and $3$ and all Galois extensions $F / L$, assuming the parity conjecture for twists holds for all self-dual representations of $\Gal(F / L)$. 
\end{theorem}

\begin{example}\label{ex-d42-3}
	We revisit Example \ref{ex-d42}, where $G = \Gal(F / \bQ) = D_{21}$.
	Let $\rho_{21}$ be a faithful irreducible representation of $G$. Then $\bQ(\rho_{21}) = \bQ(\zeta_{21} + \zeta_{21}^{-1})$ with a unique quadratic subfield $\bQ(\sqrt{21}) \subset \bQ(\rho_{21})$. The sum of $\rho_{21}$ and its Galois conjugates forms a rational irreducible representation $\sigma_{21}$, which satisfies
	$$ \sigma_{21} = \bC[G / C_2] \ominus \bC[G / D_{7}] \ominus \bC[G /S_3 ] \oplus \bC[G / G].$$
Therefore the norm relations test predicts that if 
$$ \frac{C_{E / F^{C_2}}\cdot C_{E / \bQ}}{C_{E / F^{D_7}}\cdot C_{E / F^{S_3}}}$$
is not the norm of an element in $\bQ(\sqrt{21})$, then $E$ has positive rank over $F$. 
\end{example}

\subsection{Tamagawa numbers as regulator constants}
Using the notation of Conjecture \ref{nrt}, the primary focus of this paper is in expressing 
$$ \prod_i C_{E / F^{H_i}}\Big{/}\prod_j C_{E / F^{H_j'}}  \mod N_{\bQ(\sqrt{D})/\bQ}(\bQ(\sqrt{D})^{\times})$$
as a product of group-theoretic invariants known as {\em regulator constants}, with exponents determined by twisted root numbers. When this product fails to be a norm from $\bQ(\sqrt{D})$, this implies that there is a twisted root number equal to $-1$, yielding Theorem \ref{thm:weak}.
This provides a stronger conclusion than Conjecture \ref{nrt}, which only conjectures that $\rk E/ F$ should be positive. 

Let $G$ be a finite group and consider the formal sum of subgroups $\Theta = \sum_i n_i H_i$ with $n_i \in \bZ$. Suppose that 
\begin{equation}\label{k-rel}
	\bigoplus_i \bC[G / H_i]^{\oplus n_i} = \rho \oplus \rho^{\sigma}\tag{$\dagger$}
\end{equation}
where $\rho$ is a representation of $G$ with $\bQ(\rho) / \bQ$ quadratic and $\langle \sigma \rangle = \Gal(\bQ(\rho) / \bQ)$. To each rational representation $\tau$ of $G$ one can associate the regulator constant $\cC_{\Theta}(\tau) \in  \bQ^{\times}$, see Definition \ref{def:regulator_constant} below. This is the determinant of a rational pairing with respect to a rational basis of $\tau$, and the value of $\cC_{\Theta}(\tau) \mod N_{\bQ(\rho) / \bQ}(\bQ(\rho)^{\times})$ is independent of all choices (these choices and further details on regulator constants are discussed in \S\ref{subsec-regs}).

The following theorem establishes the use of regulator constants in relating Tamagawa numbers and local fudge factors to root numbers. 

\begin{theorem}[{Theorem \ref{corr:main}}]\label{thm:intro}
Let $E$ be an elliptic curve over a number field $L$, with semistable reduction at the primes above $2$ and $3$ in $L$. Let $F / L$ be a finite Galois extension with $G = \Gal(F / L)$.
Suppose that $\Theta = \sum_i n_i H_i$ satisfies \eqref{k-rel} for some representation $\rho$ of $G$ with $\bQ(\rho)$ quadratic. Then  
\begin{equation*}
\prod_i (C_{E / F^{H_i}})^{n_i} \equiv  \prod_{\tau \in \Irr_{\bQ}(G)}  \cC_{\Theta}(\tau)^{u(E / L, \chi_{\tau})}
\mod N_{\bQ(\rho) / \bQ}(\bQ(\rho)^{\times}),
\end{equation*}
where $\Irr_{\bQ}(G)$ is the set of irreducible rational representations of $G$, $\chi_{\tau}$ is a complex irreducible constituent of $\tau$, and  $u(E  / L, \chi_{\tau}) \in \{0,1\}$ satisfies $w(E / L, \chi_{\tau}) = (-1)^{u(E / L, \chi_{\tau})}$ if $\chi_\tau$ is self-dual and $u(E / L, \chi_\tau) = 0$ otherwise. 
\end{theorem}

Note that we prove this where the base field is an arbitrary number field. To prove this result, we break up $\prod_i ( C_{E / F^{H_i}})^{n_i}$ into a product of local factors for each prime $\fp$ in $L$, and prove that these local terms are given by the regulator constants of certain rational representations. These representations appear in local root number formulae given by Rohrlich \cite{RohG}, and taking the product over all primes of $L$ gives Theorem \ref{thm:intro}. 
If $\Theta$ satisfies \eqref{k-rel} with $\rho = 0$ (in this case $\Theta$ is called a {\em{Brauer relation}}),
then the congruence in Theorem \ref{thm:intro} holds up to rational squares, as proved in \cite{tamroot}, and our local compatibility statement is the same as \cite[Theorem 3.2]{tamroot}, but extended beyond Brauer relations.

\begin{example}\label{ex-d42-4}
	Consider Example \ref{ex-d42} again. The group $G = D_{21}$ has five rational representations $\{ \trivial, \epsilon, \sigma_3, \sigma_7, \sigma_{21} \}$, where $\trivial$ is the trivial representation, $\epsilon$ the sign representation, $\sigma_3$, $\sigma_7$ project to faithful irreducible rational representations of the $S_3$ and $D_{7}$-quotients respectively, and $\sigma_{21}$ is faithful. Note that $\sigma_{21}$ can be written as $\sigma_{21} = \rho \oplus \rho^{\sigma}$ where $\bQ(\rho) = \bQ(\sqrt{21})$. From the relation of permutation representations in Example \ref{ex-d42-4}, we can compute regulator constants with respect to $\Theta = C_2 - S_3 - D_7 + G$.  
On the rational irreducible representations these are 	$$ \cC_{\Theta}(\trivial)\equiv \cC_{\Theta}(\epsilon)\equiv \cC_{\Theta}(\sigma_3) \equiv 1,  \quad \cC_{\Theta}(\sigma_7)\equiv \cC_{\Theta} (\sigma_{21}) \equiv 3 \mod N_{\bQ(\rho)  / \bQ}(\bQ(\rho)^{\times}),$$  
	where $3$ is not a norm from $\bQ(\rho) $ (see Example \ref{ex:dihedral}). Thus 
$$ \frac{C_{E / F^{C_2}}\cdot C_{E / \bQ}}{C_{E / F^{D_{7}}} \cdot C_{E / F^{S_3}}} \equiv 3^{u(E / \bQ, \chi_7) + u(E / \bQ, \chi_{21})}  \mod N_{\bQ(\rho) / \bQ}(\bQ(\rho) ^{\times}),$$
where $\chi_7$, $\chi_{21}$ are complex irreducible constituents of $\sigma_{7}$, $\sigma_{21}$, respectively, and $w(E / \bQ, \chi_7) = (-1)^{u(E / \bQ, \chi_7)}$, $w(E / \bQ, \chi_{21}) = (-1)^{u(E / \bQ, \chi_{21})}$. It follows that if the norm relations test predicts that $\rk E / F > 0$ using the relation $\Theta$, then $w(E / \bQ, \chi_7) = -1$, or $w(E / \bQ, \chi_{21}) = -1$, but not both. 
\end{example}

\begin{remark}\label{two-and-three}
	To prove Theorem \ref{thm:intro}, we need to be able to compute the $C_{E / F^{H_i}}$, and so be able to compute Tamagawa numbers and understand how they vary in field extensions. 
Since these are much more difficult to compute at primes of additive reduction above $2$ and $3$, we assume that $E$ is semistable at these primes. 
\end{remark}

The naming of regulator constants comes from the fact that they can be used to compute products of regulators of elliptic curves. Let $F / L$ be a Galois extension with $G = \Gal(F / L)$, and let $E / L$ be an elliptic curve. Let $\Theta = \sum_i n_i H_i$ be a formal sum of subgroups of $G$ satisfying \eqref{k-rel} for some representation $\rho$ with $\bQ(\rho) / \bQ$ quadratic. If $\langle E(F) \otimes_{\bZ} \bQ, \rho \rangle = 0$, then 
$$ \cC_{\Theta}(E(F) \otimes_{\bZ} \bQ) \equiv \prod_i (\Reg_{E / F^{H_i}})^{n_i} \mod N_{\bQ(\rho) / \bQ}(\bQ(\rho)^{\times}).$$
If we additionally assume that $E$ has semistable reduction at the primes above $2$ and $3$, and that the parity conjecture for twists of $E$ by self-dual representations of $G$ holds, then Theorem \ref{thm:intro} and properties of regulator constants imply that 
\begin{equation}\label{c-reg}
	\prod_i (C_{E / F^{H_i}})^{n_i} \equiv \prod_i (\Reg_{E / F^{H_i}})^{n_i} \mod N_{\bQ(\rho) / \bQ}(\bQ(\rho)^{\times})\tag{$\dagger\dagger$}
\end{equation}
(see Corollary \ref{corr:parity}). In the context of the norm relations test, where $L = \bQ$, this shows that if the test predicts positive rank,  
then the regulator of $E$ over some intermediate subfield of $F / \bQ$ is non-trivial, and so $E$ must have positive rank over this subfield.  

Establishing the congruence \eqref{c-reg} is closely related to \cite[Proposition 27]{dok-wier-ev}, where a similar result is proved under the assumption that the base field is $\bQ$. However, in that setting, the result extends to more general abelian extensions, rather than being restricted to quadratic extensions.  
Proposition \cite[Proposition 27]{dok-wier-ev} follows from the same conjectural properties of twisted $L$-functions necessary for Conjecture \ref{nrt} to hold.  
On the other hand, by relating the Tamagawa numbers and local fudge factors to root numbers, our result follows from the parity conjecture for twists and is independent of properties of $L$-functions.  
\subsection{K-relations and regulator constants}\label{sec:rep_intro}

To conclude the introduction, we highlight the representation-theoretic structures that underpin our work. 

\subsubsection{K-relations and functions on subgroups}
Let $G$ be a finite group, and let $K / \bQ$ be a finite Galois extension. A formal sum $\Theta = \sum_i n_i H_i$ of subgroups $H_i \leq G$ with $n_i \in \bZ$ is called a {\em $K$-relation for $G$} if there exists a representation $\rho$ of $G$ with $\bQ(\rho) \subset K$ and an isomorphism  
$$ \bigoplus_i \bC[G /H_i]^{\oplus n_i}\ 
= \bigoplus_{\sigma \in \Gal(K / \bQ)} \rho^{\sigma},$$
as virtual representations. Note that the isomorphism still holds if any $H_i$ is replaced by a conjugate subgroup. By Lemma \ref{lem:smaller-rel}, if $\Theta$ is a $K$-relation for $G$, it is also a $K'$-relation for all subfields $K' \subset K$ with $K ' / \bQ$ Galois.
The combination of subgroups appearing in the norm relations test is a $\bQ(\rho)$-relation for $G$ where $\rho$ is some irreducible representation of $G$, as well as a $\bQ(\sqrt{D})$-relation for every quadratic subfield $\bQ(\sqrt{D}) \subset \bQ(\rho)$ (when such a subfield exists). 

In the case where 
$\rho = 0$, the relation reduces to
$$\bigoplus_i \bC[G / H_i]^{\oplus n_i} = 0,$$
which is known as a {\em {Brauer relation}}. Identifying subgroups of $G$ (up to conjugacy) with finite transitive $G$-sets, Brauer relations can be considered as differences $X - Y$ of non-isomorphic $G$-sets $X$, $Y$ with $\bC[X] = \bC[Y]$.  
Brauer relations have many uses in number theory.  
For instance, they have been used to study class numbers and unit groups \cite{german-brauer}, \cite{smit}. 
In the context of elliptic curves, they have been used by the Dokchitser brothers \cite{annals}, \cite{sweep} to deduce the
parity conjecture for elliptic curves $E / L$, assuming finiteness of the $2$- and $3$-primary parts of $\sha_{E / L(E[2])}$. In \cite{tamroot}, similar ideas are used to deduce the $p$-parity conjecture for twists of certain representations. We will adopt a very similar approach to \cite{tamroot} in  
 applying $K$-relations to study rank predictions via the norm relations test. 

To reinterpret the computation of invariants over products of subfields associated with a given relation, we study functions defined on subgroups of $G$. 
Consider a function $\phi$ defined on subgroups $H \leq G$ that maps to some abelian group $\cA$ (written multiplicatively) and satisfies $\phi(H) = \phi(H')$ for conjugate subgroups $H$, $H'$. Then $\phi$ may also be viewed as a function on finite $G$-sets. If 
$ \prod_i \phi(H_i)^{n_i} = 1$
whenever $\Theta = \sum_i n_i H_i$ is a Brauer relation, we say $\phi$ is {\em representation theoretic}. This means that the value of $\phi$ on a $G$-set $X$ only depends on $\bC[X]$. 
In the case where $\cA = \bQ^{\times}$, we say that $\phi$ is {\em trivial on $K$-relations for $G$} if $\phi(\Theta) \in N_{K / \bQ}(K^{\times})$ whenever $\Theta$ is a $K$-relation for $G$. Such functions are representation theoretic when considered as functions to 
$\bQ^{\times} / N_{K / \bQ}(K^{\times})$. 

Functions on subgroups that arise from Galois-equivariant functions on representations of $G$ provide natural examples of functions that are trivial on $K$-relations for $G$. Let $G$ be of exponent $n$, and consider any $\bQ(\zeta_n)$-valued function $g$ on complex representations of $G$ that satisfies 
$$ g(\chi \oplus \psi) = g(\chi)g(\psi), \quad g(\chi^{\sigma}) = g(\chi)^{\sigma},$$
for all representations $\chi, \psi$ of $G$ and all $\sigma \in \Gal(\bQ(\zeta_n) / \bQ)$. Then the function on subgroups of $G$ given by $\phi(H) = g(\bC[G / H])$ is both representation-theoretic and trivial on $K$-relations for $G$. 

If $G = \Gal(F / \bQ)$, and $E / \bQ$ is an elliptic curve, the map
$$ H \leq G \mapsto C_{E / F^{H}} $$
is constant on conjugate subgroups.
The norm relations test takes a particular $\bQ(\sqrt{D})$-relation $\Theta = \sum_i n_i H_i$, and tests whether 
$$ \prod_i (C_{E / F^{H_i}})^{n_i} = 1 \in \bQ^{\times} / N_{\bQ(\sqrt{D}) / \bQ}(\bQ(\sqrt{D})^{\times}). $$  
If the map $H \mapsto C_{E / F^{H}}$ is trivial on $\bQ(\sqrt{D})$-relations for $G$, then the norm relations test cannot be used to predict positive rank for this choice of elliptic curve and Galois extension. In \S\ref{sec:rep_theory}, we analyse in detail when functions (or ratios of functions) are trivial on $K$-relations.  

\subsubsection{Regulator constants}\label{subsec-regs}
Regulator constants were first introduced in \cite{annals, tamroot}. Let $G$ be a finite group and $\Theta$ a Brauer relation for $G$. 
Given a rational representation $\tau$ of $G$ equipped with a non-degenerate $G$-invariant bilinear pairing, its regulator constant is a rational number $\cC_{\Theta}(\tau) \in \bQ^{\times} $ (defined in \cite[\S2.ii]{tamroot}). This is a purely representation-theoretic object that is computed as a product of determinants of the pairing on invariant subspaces of $\tau$. The value of $\cC_{\Theta})(\tau) \mod \bQ^{\times 2}$ is independent of the choice of pairing and bases for the invariant subspaces.  

We extend the definition of regulator constants to the case of $K$-relations. Let $\Theta = \sum_i n_i H_i$ be a $K$-relation for $G$, and let $\tau$ be a $\bQ G$-representation. The regulator constant $\cC_{\Theta}(\tau)$ of $\tau$ with respect to $\Theta$ is defined in direct analogy with the case of Brauer relations (see
Definition \ref{def:regulator_constant}). A priori, this definition depends on the choice of non-degenerate $G$-invariant bilinear pairing on $\tau$, but ---as in the Brauer relation case--- should be independent of this choice. Independence holds (up to squares in $\bQ^{\times}$) whenever  
$\sum_i n_i \langle \tau, \bC[G / H_i]\rangle = 0$. In the general situation, some extra care is required: by restricting to rational pairings and working modulo the subgroup $N_{K / \bQ}(K^{\times})\cdot \bQ^{\times 2}$, one recovers independence (see Proposition  
\ref{prop:indep}). 
Computing with respect to some rational pairing on $\tau$, the regulator constant $\cC_{\Theta}(\tau)$ is a non-zero rational number whose value mod $N_{K / \bQ}(K^{\times}) \cdot \bQ^{\times 2}$ is independent of this choice of pairing.  

In the case of Brauer relations, regular constants have been applied to various problems in elliptic curves and number theory. 
They have been used to establish the $p$-parity conjecture for elliptic curves over $\bQ$ \cite{annals}, to study Selmer group growth in certain extensions \cite{bartel-selmer}, and to analyse Galois module structures in the context of unit groups and higher $K$-groups of number fields \cite{bartel-dihedral, bartel-smit}. More recently, they have also appeared in the study of covers of curves, particularly in the context of the parity of ranks of Jacobians \cite{dgkm}. While we have not yet investigated potential applications of regulator constants for $K$-relations beyond this paper, this could be an interesting direction for future investigation.  

\subsection{Layout}
This paper is structured as follows.
In Section $2$, we develop the theory of $K$-relations and regulator constants, introducing the key representation-theoretic tools needed for our main results. We also consider functions on the Burnside ring of a group, and their behaviour when evaluated at $K$-relations. This section is self-contained and purely representation-theoretic, making it accessible to those interested in $K$-relations independently of their application to elliptic curves. 
In Section 3, we establish the main technical results, linking regulator constants to root numbers. 
In Section 4, we apply these results to our study of the norm relations test. We give examples of when the norm relations test cannot predict positive rank. We also consider the ratio of BSD-quotients corresponding to a $K$-relation. 

\subsection{Notation}\label{sec:notation}
Given a perfect field $K$, let $\overline{K}$ denote its algebraic closure, and $G_K$ the absolute Galois group $\Gal(\overline{K} / K)$. For a place $v$ of $K$, let $K_v$ be the completion of $K$ with respect to $v$.

Let $E' / \cL$ be an elliptic curve defined over a local field $\cL$, and let $E / L$ be an elliptic curve defined over a number field $L$. We use the following notation for elliptic curves. 
\begin{table}[H]
	\begin{tabular}{l l }
		$E(F)_{\bC}$ & = $E(F) \otimes_{\bZ} \bC$, the complex representation of $G = \Gal(F / L)$ for a finite Galois \\ & extension of number fields,\\
		$w(E' / \cL)$ & local root number of $E' / \cL$,\\
		$w(E / L)$ & global root number $\prod_{v \leq \infty} w(E / L_v)$,\\
		$w(E' / \cL, \rho)$ & local twisted root number of $E' / \cL$ associated to a
		self-dual Artin representation \\ & over $\cL$, \\
		$w(E / L, \rho)$ & global twisted root number associated to a self-dual Artin representation $\rho$ over \\ & $L$, given by $\prod_{v \leq \infty} w(E / L_v, \Res_{D_v} \rho)$, with $D_v = G_{L_v} \leq G_L$, \\
		$u(E / L, \rho)$ & $\in \{0,1\}$ satisfying $w(E / L, \rho) = (-1)^{u(E / L, \rho)}$ if $\rho$ is self-dual, and $u(E / L, \rho)  = 0$ otherwise,  \\
	$ c(E' / \cL)$ & Tamagawa number of $E' / \cL$ when $\cL$ is non-Archimedean, \\
	$ c_v(E / L)$ & local Tamagawa number $c(E / L_v)$ of $E / L$ at a finite place $v$.\\
\end{tabular}
\end{table}
Fix an invariant differential $\omega$ on $E$. For a finite place $v$ of $L$, define
$$
C_v(E / L, \omega) =
		c(E / L_v) \cdot \left| \omega / \omega^0\right|_{L_v}, 
$$
where $|\cdot|_{L_v}$ denotes the normalised absolute value on $L_v$, and $\omega^0$ is a N\'{e}ron differential for $E / L_v$. By $\omega/ \omega^0$ we mean the unique $v$-adic number $\delta$ with $\omega = \delta \omega^0$. The term $C(E' / \cL, \omega)$ is defined in the same way, where $\omega$ is a non-zero differential on $E' / \cL$.
We define
$$ C_{E / L} = \prod_{v < \infty} C_v(E / L, \omega),$$
for a choice of invariant differential $\omega$ on $E$, taking the product over all finite places. If $L = \bQ$ we choose $\omega$ to be a global minimal differential, so that our definition coincides with the one given in \cite[Notation 17]{dok-wier-ev}. 

Note that this depends on $\omega$.  
Whenever we consider a product of $C_{E / L'}$ for fixed $E / L$ and varying field extensions $L'/L$, we compute each $C_{E / L'}$ with respect to the same differential $\omega$ on $E / L$. \\ 

The following notation relates to the representations of a finite group $G$ and its Burnside ring. 
\begin{table}[H]
\centering
\begin{tabular}{l l}
	$\Rep(G)$ & the complex representation ring of $G$,\\
	$\Irr_{\bC}(G)$, $\Irr_{\bQ}(G)$ & set of irreducible complex and irreducible rational representations of $G$ respectively,\\
	$\hat{C}(G)$ & $= \Char_{\bQ}(G) / \Perm(G)$, the quotient of the ring of representations with $\bQ$-valued \\ & characters by the ring of virtual permutation representations (Definition \ref{def:c_hat}), \\
	$\bQ(\rho)$ & the abelian extension of $\bQ$ generated by the character values of $\rho \in \Rep(G)$, \\
	$N_{K / \bQ}(\rho)$ & $= \oplus_{\sigma} \rho^{\sigma}$ where $\sigma$ ranges over the elements of $G_{\bQ}$ that factor through $\Gal(K / \bQ)$\\ & (Definition \ref{def:rep_norm}), \\ 
	$\B(G)$ & the Burnside ring of $G$ (Definition \ref{def:burnside}),\\ 
	$\Theta = \sum_i n_i H_i$ & a $K$-relation for a number field $K$, with $n_i \in \bZ$, $H_i \leq G$ (Definition \ref{def:k_rel}), \\
	$\bC[G / \Theta]$ & $ = \oplus_i \bC[G / H_i]^{\oplus n_i}$ for $\Theta = \sum_i n_i H_i \in \B(G)$, \\ 
	$\cC_{\Theta}(\tau)$ & regulator constant for a self-dual representation $\tau$ of $G$ (Definition \ref{def:regulator_constant}),\\
	$\chi_{\tau}$ & any complex irreducible constituent of $\tau \in \Irr_{\bQ}(G)$,\\ 
	$D_n$ & the dihedral group of order $2n$. 
\end{tabular}
\end{table}

\begin{acknowledgements}
	I wish to express my gratitude to Vladimir Dokchitser for his helpful discussions and guidance throughout the development of this paper. I’d also like to thank Albert Lopez-Bruch for his input during the initial stages of the work, and the anonymous referee for valuable feedback on an earlier version.  This work was supported by the Engineering and Physical Sciences Research Council
[EP/S021590/1], the EPSRC Centre for Doctoral Training in Geometry and Number
Theory (The London School of Geometry and Number Theory) at University College
London.

\end{acknowledgements}

\section{K-relations and functions on the Burnside ring}\label{sec:rep_theory}

This section concerns the representation-theoretic concepts that are crucial in proving the main results of this paper in \S\ref{sec:main}. We have kept this section self-contained, as it may be of independent interest and have applications beyond elliptic curves. The primary objective is to extend the theory introduced in \cite[\S2]{tamroot} so that it can also be applied in the setting of $K$-relations. For ease of reference, we recall relevant definitions and results from \cite[\S2]{tamroot} throughout. 

We first study K-relations as introduced in the introduction, then examine functions on the Burnside ring of a finite group and their behaviour when evaluated on K-relations. Finally, we define and study regulator constants for K-relations. 

Throughout this section, $G$ is a finite group, and $K / \bQ$ is a finite Galois extension.

\subsection{K-relations}

Given a finite $G$-set $X$, let $[X]$ denote its isomorphism class.

\begin{definition}\label{def:burnside}
	The {\em Burnside ring} $\B(G)$ is the ring of formal linear combinations of isomorphism classes of finite $G$-sets, modulo the relations $$[S] + [T] = [S \sqcup T], \quad [S]\cdot[T] = [S \times T]$$ for $G$-sets $S$, $T$ ($S \sqcup T$ being the disjoint union and $S \times T$ the Cartesian product).  
\end{definition}

We will not need the multiplicative structure of the Burnside ring. As isomorphism classes of finite transitive $G$-sets are in bijection with subgroups of $G$ up to conjugacy, we will write an element of $\B(G)$ as the formal sum $\Theta = \sum_i n_i H_i$ with $n_i \in \bZ$, $H_i \leq G$.

Let $\Rep(G)$ be the complex representation ring of $G$. 
There is a map from the Burnside ring to $\Rep(G)$ given by sending a $G$-set $X$ to $\bC[X]$.  
Given $\Theta = \sum_i n_i H_i \in \B(G)$, the image of $\Theta$ under this homomorphism is a formal $\bZ$-linear combination of isomorphism classes of representations, which we denote by $$\bC[G / \Theta] := \bigoplus_i \bC[G / H_i]^{\oplus n_i }.$$ 

We define restriction and induction maps on Burnside rings corresponding to restriction and induction of the corresponding permutation representations.

\begin{definition}
For $D \leq G$ define the maps $\Res_D^G \colon \B(G) \to \B(D)$, $\Ind_D^G \colon \B(D) \to \B(G)$ by 
$$
\Res_D^G H = \sum_{x \in H \backslash G / D} D \cap H^{x^{-1}}, 
\quad \Ind_D^G H' = H',
$$
for $H \leq G$, $H' \leq D$. 
\end{definition}

Clearly $\Ind_D^G \bC[D / H'] = \bC[G / \Ind_D^G H']$, and by Mackey's decomposition 
$$\Res_D^{G} \bC[G / H] =  
\bigoplus_{x \in H \backslash G / D} \bC[D / D \cap H^{x^{-1}}] = \bC[D / \Res^G_D H].$$ 
We also write these functions as $\Res_D$, $\Ind_D$ when the group $G$ is clear.  

Let $\rho \in \Rep(G)$ a representation of $G$. The field generated by the character values of $\rho$ is a finite algebraic extension of $\bQ$, which we denote by $\bQ(\rho)$. 

\begin{definition}\label{def:rep_norm}
Let $K / \bQ$ be a finite Galois extension. For $\rho \in \Rep(G)$ with $\bQ(\rho) \subset K$, define  
$$
N_{K  / \bQ}(\rho) = \bigoplus_{\sigma \in \Gal(K  / \bQ)} \rho^{\sigma}, 
$$
where $\rho^{\sigma}$ is the representation of $G$ with $\Tr \rho^{\sigma}(g) = \sigma(\Tr \rho(g))$ for all $g \in G$. 
\end{definition}

Note that $N_{K / \bQ}(\rho) = \bigoplus_{\tau \in \Gal(K / \bQ)} \rho^{\tau} = N_{\bQ(\rho) / \bQ}(\rho)^{[ K \colon \bQ(\rho)]}$. 

\begin{definition}\label{def:k_rel}
Let $G$ be a finite group, and let $K / \bQ$ be a finite Galois extension. An element $\Theta \in \B(G)$ is a {\em $K$-relation for $G$}, if there exists $\rho \in \Rep(G)$ with $\bQ(\rho) \subset K$ and 
$$\bC[G / \Theta] =  N_{K  / \bQ}(\rho).
$$
\end{definition}

We will often refer to these as $K$-relations if the group $G$ involved is clear. The following are some elementary properties of $K$-relations.

\begin{lemma}\label{lem:smaller-rel}
Let $G$ be a finite group, $K$ a finite Galois extension of $\bQ$. If $\Theta \in \B(G)$ is a $K$-relation for $G$, then $\Theta$ is also a $L$-relation for $G$ for all $\bQ \subset L \subset K$ such that $L / \bQ$ is Galois.
\end{lemma}

\begin{proof}
	If $\bC[G / \Theta] = N_{K /\bQ}(\rho)$ with $\bQ(\rho) \subset K$, then $\tau = \bigoplus_{\sigma' \in \Gal(K / L) } \rho^{\sigma'}$ satisfies $\bQ(\tau) \subset L$ and $N_{K / \bQ}(\rho) = N_{L / \bQ}(\tau)$.  
\end{proof}

\begin{proposition}\label{prop:K_rels_properties}
Let $\Theta = \sum_i n_i H_i \in \B(G)$ be a $K$-relation for $G$. Then
\begin{enumerate}
\item (Sum) If $\Psi \in \B(G)$ is another $K$-relation for $G$, then $\Theta \pm \Psi$ is a $K$-relation for $G$, 
\item (Restriction) If $H \leq G$ then $\Res^{G}_{H} \Theta \in \B(H)$ is a $K$-relation for $H$,
\item (Induction) If $G \leq G'$ then $\Ind^{G'}_{G} \Theta \in \B(G')$ is a $K$-relation for $G'$,
\item (Projection) If $N \triangleleft G$ then $N \cdot \Theta / N = \sum_i n_i H_i N / N \in \B(G / N)$ is a $K$-relation for $G / N$, 
\item (Lifting) If $N \subset H_j'$ and $\sum_j m_j H_j' N / N \in \B(G / N)$ is a $K$-relationfor $G / N$, then $\sum_j m_j H_j'$ is a $K$-relation for $G$. 
\end{enumerate}
\end{proposition}

\begin{proof}
	Let $\bC[G / \Theta] =  N_{K / \bQ}(\rho)$, $\bC[G / \Psi] = N_{K / \bQ}(\tau)$ for $\rho$, $\tau \in \Rep(G)$. Then $\bC[G / \Theta \pm \Psi] = N_{K / \bQ}(\rho \pm \tau)$, 
$\bC[H / \Res_{H}^G \Theta] = N_{K / \bQ}(\Res_H^G \rho)$, and
$\bC[G' / \Ind_G^{G'}\Theta] = N_{K / \bQ}(\Ind_{G}^{G'} \rho)$, proving $(1)$, $(2)$, $(3)$.  
One has $\bC[G / H_i]^{N} = \bC[G / N H_i]$ as $G$-representations (cf. \cite[Theorem 2.8]{tamroot}). Thus 
$\bC[G  / \Theta]^{N} = N_{K / \bQ}(\rho^N) = \bC[G / N / N\cdot \Theta / N]$, implying $(4)$. For $(5)$, it is clear that the relation can be lifted to $G$.   
\end{proof}

The previous proposition shows that the space of $K$-relations for $G$ forms a sublattice of the Burnside ring $\B(G)$.  
If $n = [K : \bQ]$ and $\Theta \in \B(G)$, then
\[
	\bC[G / n\Theta] = \bC[G / \Theta]^{\oplus n} =  N_{K / \bQ}(\bC[G / \Theta]),
\]
and hence the sublattice generated by $K$-relations has finite index in $\B(G)$ dividing $[K \colon \bQ]$.

\begin{remark}[Brauer relations]\label{remark:brauer}
	When $\rho = 0$, $\Theta$ is called a {\em Brauer relation} (cf. \cite[\S 2.1]{tamroot}) . Non-zero elements of $\Theta \in \B(G)$ with $\bC[G / \Theta] =  0$ are instances of non-isomorphic $G$-sets with isomorphic permutation representations. Indeed, writing $\Theta = X - X'$ as a difference of finite $G$-sets, one has $X \not\simeq X'$ but $\bC[X] =  \bC[X']$. The rank of the space of Brauer relations in $\B(G)$ is equal to the number of non-cyclic subgroups of $G$ up to conjugacy (\cite[Remark 2.5]{tamroot}). Brauer relations have been classified by Bartel and Tim Dokchitser in \cite{Brauer}. 
\end{remark}

\begin{notation}
There is an inclusion of rings 
$$ \Perm(G) \subset R_{\bQ}(G) \subset \Char_{\bQ}(G), $$
where $\Perm(G)$ is the ring of virtual permutation representations, $R_{\bQ}(G)$ the ring of rational representations, and $\Char_{\bQ}(G)$ the ring of representations with $\bQ$-valued characters. These rings have the same rank as $\bZ$-modules. A basis of $\Char_{\bQ}(G)$ is given by the distinct $N_{\bQ(\chi) / \bQ}(\chi)$ as $\chi$ ranges over the complex irreducible representations of $G$. 
\end{notation}

\begin{definition}\label{def:c_hat}
Define the group $$\hat{C}(G) = \Char_{\bQ}(G) / \Perm(G).$$ 
\end{definition}

\begin{remark}\label{rem:c_hat_finite}
It follows from Artin's induction theorem that $\hat{C}(G)$ is finite with exponent dividing $|G|$. When $G$ is abelian or dihedral, $\hat{C}(G) = 1$. The order of $N_{\bQ(\chi) / \bQ}(\chi)$ in $\Char_{\bQ}(G) / R_{\bQ}(G)$ is the Schur index $m_{\bQ}(\chi)$ of $\chi$. The group $R_{\bQ}(G) / \Perm(G)$ is not trivial in general (cf. \cite[Exc. 13.4]{Serre}). It is trivial when $G$ is a $p$-group (\cite{Segal}). Given a finite group $G$, the structure of $\hat{C}(G)$ can be computed using the results of \cite{Bartel}.     
\end{remark}

The finiteness of $\hat{C}(G)$ implies that for any $\bC G$-representation $\rho$, there exists an integer $m$ and $\Theta \in \B(G)$ such that $ \bC[G / \Theta] = N_{\bQ(\rho) / \bQ}(\rho)^{\oplus m}$.  
The following proposition describes the decomposition of the permutation representation associated to a $K$-relation into irreducible rational representations.  

\begin{proposition}\label{prop:norm_rel_irrs}
Let $\Theta \in \B(G)$ be a $K$-relation. Then
$$
\bC[G / \Theta] = \bigoplus_{\tau \in \Irr_{\bQ}(G)} N_{\bQ(\chi_{\tau}) / \bQ}(\chi_{\tau})^{\oplus k_{\tau}}
$$
where $k_{\tau} \geq 1 $ is divisible by $[K \colon K \cap \bQ(\chi_{\tau})]$. The sum ranges over the rational irreducible representations of $G$, and $\chi_{\tau}$ denotes a complex irreducible constituent of $\tau \in \Irr_{\bQ}(G)$. 
\end{proposition}

\begin{proof}
Suppose that $\bC[G / \Theta] = N_{K / \bQ}(\rho)$ for some representation $\rho$ of $G$ with $\bQ(\rho) \subset K$. 
Let $L_{\tau} = K \cap \bQ(\chi_{\tau})$. A basis of representations of $G$ with $K$-valued character is given by $$\{ \varphi_{\tau} = \hspace{-0.5em} \bigoplus_{ \sigma \in \Gal(\bQ(\chi_{\tau}) / L_{\tau})} \hspace{-0.5em}\chi_{\tau}^{\sigma} \colon \tau \in \Irr_{\bQ}(G) \}.$$ Thus $\rho = \bigoplus_{\tau} \varphi_{\tau}^{\oplus n_{\tau}}$ for some $n_{\tau} \in \bZ$ and so  
$$ 
N_{K / \bQ}(\rho) = \hspace{-0.5em} \bigoplus_{\tau \in \Irr_{\bQ}(G)}\hspace{-0.5em}  N_{K / \bQ}(\varphi_{\tau})^{\oplus n_{\tau}} = \hspace{-0.5em} \bigoplus_{\tau \in \Irr_{\bQ}(G)} \hspace{-0.5em} N_{L_{\tau} / \bQ}(\varphi_{\tau})^{\oplus  n_{\tau}  [K \colon L_{\tau}]}
= \hspace{-0.5em} \bigoplus_{\tau \in \Irr_{\bQ}(G)} \hspace{-0.5em} N_{\bQ(\chi_{\tau}) / \bQ}(\chi_{\tau})^{\oplus n_{\tau} [K \colon L_{\tau}]}.
$$
\end{proof}

\begin{example}[Cyclic Groups]\label{ex:cyclic_groups1}
	Let $G = C_n$. For $d \mid n$, let $\chi_d$ be an order $d$ character of $G$ with $\bQ(\chi_d) = \bQ(\zeta_d)$. Then 
$$\bC[G / \Psi_d] = N_{\bQ(\chi_d) / \bQ}(\chi_d), \qquad \Psi_d = \sum_{d' \mid d} \mu(d/ d') C_{n / d'} \in \B(G),$$
where $\mu$ is the M\"{o}bius function. 
Moreover $\Psi_d$ is the unique element of $\B(G)$ with $\bC[G / \Psi_d] = N_{\bQ(\chi_d) / \bQ}(\chi_d)$. Indeed if $\Psi' \in \B(G)$ also satisfies $\bC[G / \Psi'] = N_{\bQ(\psi_d) / \bQ}(\psi_d)$, then $\Psi_d - \Psi'$ is a Brauer relation. But cyclic groups have no Brauer relations (cf. \cite[Exercise 2.2]{tamroot}) so $\Psi_d = \Psi'$.

If $\Theta$ is a $K$-relation for $G$, then by Proposition \ref{prop:norm_rel_irrs} $\Theta = \sum_{d \mid n} a_d \Psi_d$ where $a_d \in \bZ_{\geq 1}$ is divisible by $[K \colon K \cap \bQ(\zeta_d)]$. \end{example}

\subsection{Functions on the Burnside ring}

This subsection concerns functions from the Burnside ring to an abelian group.
The content of this section appears in \cite[\S2.iii]{tamroot} and we repeat it for convenience of reference.  

Let $A$ be an abelian group. We consider linear functions $f \colon \B(G) \to A$ satisfying $f(X \sqcup Y) = f(X)f(Y)$ for finite $G$-sets $X$, $Y$. 
In applications to number theory, such functions appear naturally when $G$ is the Galois group of an extension of number fields and one is interested in computing invariants associated to subfields. 

\begin{example}\label{ex:C}
	Let $E$ be an elliptic curve over a number field $L$, and $F / L$ a finite Galois extension with $G = \Gal(F / L)$.
	If $H , H' \leq G$ are conjugate subgroups, then $F^{H} \simeq F^{H'}$ and so $C_{E / F^{H}} = C_{E / F^{H'}}$ ($C_{E / F^{H}}$ defined in \S\ref{sec:notation}).  Thus the function $C \colon \B(G) \to \bQ^{\times}$ sending $H \leq G$ to $C_{E / F^{H}}$ is well-defined.  
\end{example}

Often such functions will depend on the decomposition group at a prime of the base field of the Galois extension, which motivates the following definition. 

\begin{definition}[{\cite[Definition 2.33]{tamroot}}]
	For $D \leq G$, say a function $f \colon \B(G) \to A$ is {\em $D$-local} if there exists a function $f_D$ on $\B(D)$ such that $f(H) = f_D(\Res_D H)$ for all $H \leq G$. If this is the case, write $$f = (D, f_D).$$  
\end{definition}

\begin{example}
	Let $G = \Gal(F / L)$ be the Galois group of an extension of number fields. Let $\fp$ be a prime of $L$, and $\fq$ a prime of $F$ above $\fp$ with decomposition group $D_{\fq}$. For $H \leq G$, the primes above $\fp$ in $F^H$ are in one-to-one correspondence with elements of $H \backslash G / D_{\fq}$. For $\lambda \in \bQ^{\times}$, the function $f \colon \B(G) \to \bQ^{\times }$, $$f(H) = \lambda^{\# \{ \text{primes above $\fp$ in $F^{H}$} \} } = \lambda^{| H \backslash G / D_{\fq}|},$$ that counts the primes above $\fp$ is a $D_{\fq}$-local function with $f = (D_{\fq}, \lambda)$.  
\end{example}

\begin{example}\label{ex2:C}
	In Example \ref{ex:C}, the function $C$ factors as a product of local functions. 
	Recall that $C_{E / F^{H}} = \prod_v C_v(E / F^{H}, \omega)$ (as defined in \S\ref{sec:notation}) for a choice of invariant differential $\omega$ on $E / L$, where $v$ ranges over all finite places of $F^{H}$. 

	For each finite place $v$ of $L$, let $D_v = \Gal(F_w / L_v) \leq G$ be the decomposition group at $v$ for a choice of place $w$ in $F$ above $v$. Let $C_{D_v} \colon \B(D_v) \to \bQ^{\times}$ be the function given by 
	$$ C_{D_v}(H') = C(E / (F_w)^{H'}, \omega),$$ 
for $H' \leq D_v$.
Then 	$ C = \prod_v (D_v, C_{D_v}),$
	with the product ranging over all finite places $v$ of $L$. Indeed, for $H \leq G$, 
	$$ C(H) = \prod_v C_{D_v}(\Res_{D_v} H) = \prod_{v} \prod_{v' \mid v} C(E / (F^{H})_{v'}, \omega) = C_{E / F^{H}}, $$
	as the double product ranges over all finite places of $F^{H}$.  
\end{example}

Let $G = \Gal(F / L)$ be the Galois group of an extension of number fields and let $D$, $I$ be the decomposition and inertia subgroups of $G$ for a place $v$ of $L$. If a place $w$ in $F^{H}$ corresponds to the double coset $Hx D$, then its decomposition and inertia groups in $F / F^{H}$ are $H \cap D^{x}$ and $H \cap I^{x}$ respectively. The ramification degree and residue degree of $w$ over $L$ is thus given by $e_{w} = \tfrac{|I|}{|H \cap I^{x}|}$ and $f_{w} = \tfrac{[D \colon I]}{[H \cap D^{x} \colon H \cap I^{x}]}$. To describe functions that depend on the ramification and residue degrees of places in intermediate fields, we introduce the following definition.

\begin{definition}[{\cite[Definition 2.35]{tamroot}}]\label{def:D_ef}
Consider $I \triangleleft D \leq G$ with $D / I$ cyclic and $\psi(e, f)$ a function of $e$, $f \in \bN$. Define a function on $\B(G)$ by 
$$
(D, I, \psi) \colon H \mapsto \prod_{x \in H \backslash G / D} \psi\left( \tfrac{|I|}{|H \cap I^{x}|} , \tfrac{[D \colon I]}{[H \cap D^{x} \colon H \cap I^x]} \right).
$$
This is a $D$-local function on $\B(G)$ with 
$$
(D, I, \psi) = \left(D, U \mapsto \psi\left(\tfrac{|I|}{|U \cap I|}, \tfrac{|D|}{|UI|}\right)\right),
$$
(since $(U \cap D) / (U \cap I) \simeq UI / I$ for $U \leq D$).
\end{definition}

The following proposition follows from the definition above.

\begin{proposition}[{\cite[Theorem 2.36]{tamroot}}]\label{prop:rename_descent}
Let $I \triangleleft D \leq G$ with $D / I$ cyclic. 
\begin{enumerate}[(i)]
	\item If $I_0 < I$ is normal in $D$ with cyclic quotient, and $\psi(e, f)$ is a function of the product $ef$, then $(D, I, \psi) = (D, I_0, \psi)$.
	\item If $I < D_0 < D$ and $\psi(e,f) = \psi(e, f / m)^{m} $ whenever $m$ divides $f$ and $[D \colon D_0]$, then $(D, I, \psi) = (D_0, I, \psi)$.
\end{enumerate}
\end{proposition}

\subsection{Functions on K-relations}

We now focus on functions from the Burnside ring to $\bQ^{\times}$ and evaluating these function on $K$-relations.

\begin{definition}
Let $f \colon \B(G) \to \bQ^{\times}$ be a linear function, $K$ a number field. Say that {\em $f$ is trivial on $K$-relations for $G$} if $f(\Theta) \equiv 1 \mod N_{K / \bQ}(K^{\times})$ whenever $\Theta$ is a $K$-relation for $G$. 
\end{definition}

If $f$ is trivial on $K$-relations, we write $f \sim_{K} 1$. If $f, g \colon \B(G) \to \bQ^{\times} $ are functions such that $f / g \sim_K 1$, we write $f \sim_K g$.

\begin{example}\label{ex:cyclic_groups2}
Let $G = C_n$, $K$ a number field and $\Theta$ a $K$-relation.  	
Consider a function $f \colon \B(G) \to \bQ^{\times}$. As in Example \ref{ex:cyclic_groups1}, $\Theta = \sum_{d \mid n} a_d \Psi_d$, where $a_d$ is divisible by $[K \colon \bQ(\zeta_d) \cap K ] $. If $f(\Psi_d) \in N_{\bQ(\zeta_d) / \bQ} \in  
 N_{\bQ(\zeta_d) / \bQ}(\bQ(\zeta_d)^{\times})$ for each $d | n$, then by Lemma \ref{lem:powernorm} $f(\Theta) \in N_{K / \bQ}(K^{\times})$. 
Therefore to show that a function $f$ on $\B(G)$ is trivial on $K$-relations for any field $K$, it is enough to show that it is trivial on the $\bQ(\zeta_d)$-relation $\Psi_d$ for each $d \mid n$.
\end{example}

\begin{proposition}\label{prop:res_quot}
Let $G$ be a group, with $D \leq G$, $N \triangleleft G$. Let $K$ be a number field, and let $f \colon \B(G) \to \bQ^{\times}$ be a linear function on $\B(G)$. 
\begin{enumerate}[(i)]
\item (Restriction) If $f = (D, f_D)$ and $f_D$ is trivial on $K$-relations for $D$, then $f$ is trivial on $K$-relations for $G$. 
\item (Quotient) Suppose $f$ is of the form $f(H) = f_{G / N}( NH / N)$ for some function $f_{G / N}$  on $\B(G / N)$. If $f_{G / N}$ is trivial on $K$-relations for $G / N$, then $f$ is trivial on $K$-relations for $G$. 
\end{enumerate}
\end{proposition}

\begin{proof}
This follows from Proposition \ref{prop:K_rels_properties}, using that if $\Theta \in \B(G)$ is a $K$-relation for $G$, then $\Res_D \Theta$ is a $K$-relation for $D$ and $N \cdot \Theta / N$ is a $K$-relation for $G / N$.  
\end{proof}

The next proposition generalises Example \ref{ex:cyclic_groups2}, and describes   
how one can prove a function is trivial on $K$-relations by reducing to irreducible representations. 

\begin{proposition}\label{prop:from_irr}
	Let $G$ be a finite group with $\hat{C}(G) = 1$, and $f \colon \B(G) \to \bQ^{\times}$ a linear function. Suppose that $f(\Psi) \in N_{K / \bQ}(K^{\times})$ whenever $\Psi$ is a Brauer relation.
	\begin{enumerate}[(i)]
	\item 
If for each $\chi \in \Irr_{\bC}(G)$, there exists $\Theta_{\chi} \in \B(G)$ with $\bC[G / \Theta_{\chi}] = N_{\bQ(\chi) / \bQ}(\chi)$ and $f(\Theta_{\chi}) \in N_{\bQ(\chi) / \bQ}(\bQ(\chi)^{\times})$, then  
$f$ is trivial on $K$-relations. 
\item
	If $K / \bQ$ is quadratic, to show that $f$ is trivial on $K$-relations, 
	it suffices to show that for each $\chi \in \Irr_{\bC}(G)$ with $K \subset \bQ(\chi)$, there exists $\Theta_{\chi} \in \B(G)$ with $\bC[G / \Theta_{\chi}] = N_{\bQ(\chi) / \bQ}(\chi)$ and $f(\Theta_{\chi}) \in N_{K / \bQ}(K^{\times})$. 
\end{enumerate}
\end{proposition}

\begin{proof}
	Let 
	$$
	\bC[G / \Theta] = N_{K / \bQ}(\rho) = \bigoplus_{\tau \in \Irr_{\bQ}(G)} N_{\bQ(\chi_{\tau}) / \bQ}(\chi_{\tau})^{\oplus k_{\tau}}
	$$
	be a $K$-relation, where $\chi_{\tau}$ is a complex irreducible constituent of $\tau \in \Irr_{\bQ}(G)$ and $[K \colon K \cap \bQ(\chi_{\tau}) ]$ divides $k_{\tau}$.
	Then $\bC[G / \Theta] = \bigoplus_{\tau \in \Irr_{\bQ}(G)} \bC[G / \Theta_{\chi_{\tau}}]^{\oplus k_{\tau}}$ and 
	$$ 
	f(\Theta) = f(\Psi)\prod_{\tau  \in \Irr_{\bQ}(G)} f(\Theta_{\chi_{\tau}})^{k_{\tau}} 
$$
where $\Psi = \Theta - \sum_{\tau \in \Irr_{\bQ}(G)} k_{\tau} \Theta_{\chi_{\tau}}$ is a Brauer relation and $f(\Psi) \in N_{K / \bQ}(K^{\times})$ by assumption.  
For $(i)$, the result follows as 
$f(\Theta_{\tau}) \in N_{\bQ(\chi_{\tau}) / \bQ}(\bQ(\chi_{\tau})^{\times})$ for all $\tau \in \Irr_{\bQ}(G)$, so that $f(\Theta_{\chi_{\tau}})^{k_{\tau}} \in N_{K / \bQ}(K^{\times})$ by Proposition \ref{prop:norm_rel_irrs} and Lemma \ref{lem:powernorm} below.
For $(ii)$, if $K \not\subset \bQ(\chi_{\tau})$ then $k_{\tau}$ is even and so $f(\Theta_{\chi_{\tau}})^{k_{\tau}} \in \bQ^{\times 2} \subset N_{K/ \bQ}(K^{\times})$. Thus 
	$$ f(\Theta) \equiv \prod_{\tau\in \Irr_{\bQ}(G), K \subset \bQ(\chi_{\tau})} f(\Theta_{\chi_{\tau}}) \mod N_{K / \bQ}(K^{\times }),$$ 
	and it follows that $f(\Theta) \in N_{K / \bQ}(K^{\times})$ from the assumption that $f(\Theta_{\chi_{\tau}}) \in N_{K / \bQ}(K^{\times})$.
\end{proof}

\begin{lemma}\label{lem:powernorm}
Let $L$ and $K$ be number fields. If $x = N_{L / \bQ}(z)$ for $z \in L$,  then $x^{[K \colon L \cap K]} \in N_{K / \bQ}(K^{\times})$. 
\end{lemma}

\begin{proof}
One has $x^{[K \colon L \cap K]} = N_{K / \bQ}(N_{L / L \cap K}(z))$. 
\end{proof}

\begin{remark}
	The proof of Proposition \ref{prop:from_irr} does not hold when $\hat{C}(G) \not= 1$. For example, the quaternion group $G = Q_8$ has $\hat{C}(G) = \bZ / 2 \bZ$, generated by the $2$-dimensional irreducible representation $\chi$. This representation has rational character. One has 
	$$ \chi^{\oplus 2} = \bC[G / C_1] \ominus \bC[G / C_2],$$
	and so $\Theta = C_1 - C_2$ is a $K$-relation for all quadratic fields $K$. However there does not exist $\Theta' \in \B(G)$ such that $\Theta = \Psi + 2 \Theta'$ for $\Psi$ a Brauer relation.  
\end{remark}

We now give some examples of functions on the Burnside ring of a cyclic group that are trivial on $K$-relations. The triviality of these functions will be essential for the proofs in \S\ref{sec:main}.

\begin{lemma}\label{lem:f_on_cyclic}
Let $G = C_n$ be cyclic. The function $g \colon H \leq G \mapsto [G \colon H]$ is trivial on $K$-relations for $G$, where $K$ is any number field.
\end{lemma}

\begin{proof}
Let $\Theta$ be a $K$-relation for $G$. Then $\Theta = \sum_{d | n } a_d \Psi_d$, where $a_d$ is divisible by $[K \colon \bQ(\zeta_d) \cap K]$ and $\Psi_d \in \B(G)$ is as in Example \ref{ex:cyclic_groups1}. One has
$$g(\Psi_d) = \prod_{d' \mid d} (d')^{\mu(d / d')}  = \Phi_d(1) = N_{\bQ(\zeta_d) / \bQ}(1 - \zeta_d), $$
where $\Phi_d$ is the $d$-th cyclotomic polynomial. Thus $g(\Psi_d) \in N_{\bQ(\zeta_d) / \bQ}(\bQ(\zeta_d)^{\times})$ and so $g$ is trivial on $K$-relations by Example \ref{ex:cyclic_groups1}.
\end{proof}

\begin{lemma}\label{lem:const_cyclic}
Let $G = C_n$ be cyclic. The function $c \colon H \leq G \mapsto \alpha$ where $\alpha \in \bQ^{\times}$ is trivial on $K$-relations for $G$, where $K$ is any number field. 
\end{lemma}

\begin{proof}
For $d \mid n$, one has $c(\Psi_d) = \alpha^{\sum_{d' | d} \mu(d / d')} = 1$ unless $d = 1$ in which case $c(\Psi_1) = \alpha$. Thus if $\Theta$ is a $K$-relation then $c(\Theta) = \alpha^{m [K \colon \bQ]} = N_{K / \bQ}(\alpha^{m})$ where $m \geq 1$. 
\end{proof}

\begin{lemma}\label{lem:alpha_beta}
Let $G = C_n$ be cyclic. Consider the function 
$$
d_{\alpha, \beta} \colon H \leq G \mapsto \begin{cases} \alpha, & \text{if } k \mid [C_n \colon H], \\ \beta, &   \text{if } k \nmid [C_n \colon H] , \end{cases}
$$
for $k \in \bN$, $\alpha, \beta \in \bQ^{\times}$. If $\alpha / \beta \in N_{\bQ(\zeta_k) / \bQ}(\bQ(\zeta_k)^{\times})$, then $d_{\alpha, \beta}$ is trivial on $K$-relations for $G$, where $K$ is any number field. 
\end{lemma}

\begin{proof}
If $k \nmid n$ or if $k = 1$ then this is the constant function which is trivial on $K$-relations by Lemma \ref{lem:const_cyclic}, so we may assume $k \mid n$ and $k > 1$. Let $\chi_k$ be a character of $G$ with $\bQ(\chi_k) = \bQ(\zeta_k)$. Then 
$$
\langle \chi_k, \bC[C_n / H] \rangle = \begin{cases} 1, & \text{if } k \mid [C_n \colon H], \\ 0, & \text{if } k \nmid [C_n \colon H] .\end{cases}
 $$
 Thus $c(H) = \alpha^{\langle \chi_k , \bQ[C_n / H] \rangle }\cdot  \beta^{1 - \langle \chi_k, \bQ[C_n / H] \rangle }$. 
 Therefore 
 $$
d_{\alpha, \beta}(\Psi_d) = (\alpha / \beta)^{\langle \chi_k, N_{\bQ(\chi_d) / \bQ} (\chi_d )\rangle } \beta^{\sum_{d' | d} \mu(d / d')} = \begin{cases} \beta, & \text{if } d = 1, \\ \alpha / \beta, & \text{if } d = k ,\\ 1, & \text{if } d \not= k, d > 1, \end{cases} 
 $$
 where $\chi_d$ is a character of $G$ of order $d$. If $\Theta$ is a $K$-relation for $G$, then $d_{\alpha, \beta}(\Theta) = \beta^{x[K \colon \bQ]}\cdot(\alpha / \beta)^{y [K \colon K \cap \bQ(\zeta_k)]}$ for some $x, y \in \bZ$. One has $\beta^{x[K \colon \bQ]} = N_{K / \bQ}(\beta^x)$, and  
$(\alpha / \beta)^{y [K \colon K \cap \bQ(\zeta_k)]} \in N_{K / \bQ}(K^{\times})$ by Lemma \ref{lem:powernorm} and the assumption that $\alpha / \beta \in N_{\bQ(\zeta_k) / \bQ}(\bQ(\zeta_k)^{\times})$. Thus $d_{\alpha, \beta}(\Theta)$ is the norm of an element of $K$. 
\end{proof}

\begin{lemma}
	Let $G = C_n$. The function 
	$$
	g \colon H \mapsto \begin{cases} [C_n \colon H],  & \text{if } 2 \mid [C_n \colon H], \\ 1, & \text{if } 2 \nmid [C_n \colon H], \end{cases}
	$$
	is trivial on $K$-relations for $G$, where $K$ is any number field.
\end{lemma}

\begin{proof}
	If $2 \nmid n$, then this is always $1$, so assume $2 \mid n$. Let $\chi_2$ be a character of $G$ of order $2$. Then $g$ can be expressed as $g(H) = [C_n \colon H]^{\langle \chi_2 ,\ \bQ[C_n / H] \rangle}$.
	For $d \mid n$, 
	$$
	g(\Psi_d) = \prod_{d' \mid d} (d')^{\mu(d / d')\cdot \langle \chi_2, \bC[C_n / C_{n / d'}] \rangle} = \begin{cases}
		2, & \text{if } d = 2, \\ p, & \text{if } d = 2 p,\ p \text{ prime,}\\ 1, & \text{otherwise}.
	\end{cases}
	$$
	In each case it is clear that $g(\Psi_d) \in N_{\bQ(\zeta_d) / \bQ}(\bQ(\zeta_d)^{\times})$ and so the result follows by Example \ref{ex:cyclic_groups1}. 
\end{proof}

If we have a function $(D, I, \psi(e,f))$ as in Definition \ref{def:D_ef} that depends only on $f$, then it can be expressed as a function of the quotient $D / I$. By Proposition \ref{prop:res_quot}, to show that $(D,I, \psi(e,f)) \sim_K 1$ it suffices to show that the function on the cyclic quotient $D / I$ is trivial on $K$-relations in $\B(D / I)$. Thus the previous results yield the following. 

\begin{corollary}\label{corr:cyclic_f_const}
Let $G$ be a finite group. Let $D \leq G$ with $D / I$ cyclic. Let $K$ be a number field.
\begin{enumerate}[(i)]
	\item The function $(D, \alpha) = (D, I ,\alpha)$ with $\alpha \in \bQ^{\times}$ is trivial on $K$-relations for $G$.
	\item The function $(D, I, f)$ is trivial on $K$-relations for $G$.
	\item The function $\left(D, I, \left\{\begin{smallmatrix} \alpha, & & \text{if }k \mid f, \\ \beta, & & \text{if } k \nmid f. \end{smallmatrix}\right.\right)$ with $\alpha, \beta \in \bQ^{\times}$ is trivial on $K$-relations for $G$ when $\alpha / \beta \in N_{\bQ(\zeta_k) / \bQ}(\bQ(\zeta_k)^{\times})$.
	\item The function $\left(D, I, \left\{ \begin{smallmatrix} f, & & \text{if } 2 \mid f, \\ 1, & & \text{if } 2 \nmid f. \end{smallmatrix} \right. \right)$ is trivial on $K$-relations for $G$. 
\end{enumerate}
\end{corollary}

\subsection{Regulator constants}

We now consider regulator constants. These are discussed in detail in \cite[\S2.ii]{tamroot}, and we will generalise some of the results therein to apply to $K$-relations. 
Throughout, $\cK$ is a field of characteristic zero.  

\begin{notation}
Let $V$ be a vector space over $\cK$, with $\langle, \rangle$ a non-degenerate $\cK$-bilinear $\cL$-valued pairing on $V$, where $\cL$ is some finite extension of $\cK$. Let $\det(\langle, \rangle | V) \in \cL^{\times} / \cK^{\times 2}$ denote $\det(\langle e_i, e_j \rangle_{i, j})$ in any $\cK$-basis $\{e_1, \ldots e_n \}$ of $V$. 
\end{notation}

\begin{definition}[Regulator constant]\label{def:regulator_constant}
Let $G$ be a finite group, $\tau$ a self-dual $\cK G$-representation. Let $\langle, \rangle$ be a $G$-invariant, non-degenerate, $\cK$-bilinear pairing on $\tau$ taking values in an extension $\cL$ of $\cK$. Let $\Theta = \sum_i n_i H_i \in \B(G)$ be a $K$-relation for $G$.  Define the {\em regulator constant}
$$
 \cC^{\langle, \rangle}_{\Theta}(\tau) = \prod_i \det\left(\frac{1}{|H_i|} \langle, \rangle \mid \tau^{H_i} \right)^{n_i} \in \cL^{\times} / \cK^{\times 2}.
$$
\end{definition}

\begin{remark}
	If $\tau$ and $\langle, \rangle$ are as in the definition, then for each $H \leq G$ the restriction of $\langle, \rangle$ to $\tau^{H}$ is non-degenerate (\cite[Lemma 2.15]{tamroot}). Thus $\det(\langle, \rangle \mid \tau^{H}) \not=0$ so $\cC_{\Theta}^{\langle, \rangle}(\tau)$ is non-zero.
\end{remark}

If $\langle \tau, \bC[G / \Theta] \rangle = 0$, $\Theta$ is a \textit{pseudo Brauer relation relative to $\tau$} as introduced in \cite[Definition 3.1]{dgkm}. As proved in \cite[Theorem 3.7]{dgkm} and stated in the proposition below, $\cC_{\Theta}^{\langle, \rangle}(\tau)$ is independent of the choice of pairing in this case. We let $\cC_{\Theta}(\tau) \in \cK^{\times}$ be a representative of the value of $\cC^{\langle, \rangle}_{\Theta}(\tau) \in \cK^{\times} / \cK^{\times 2}$ 
(we may take the pairing to be $\cK$-valued). 

\begin{proposition}\label{prop:indep}
Let $G$ be a finite group,  $\Theta = \sum_i n_i H_i \in \B(G)$ a $K$-relation for $G$. Let $\tau$ be a self-dual $\cK G$-representation, and $\langle, \rangle_1$, $\langle, \rangle_2$ two non-degenerate $\cK$-bilinear $G$-invariant pairings on $\tau$. 
Suppose that  $\langle \tau, \bC[G / \Theta] \rangle = 0$. Computing the determinants with respect to the same basis of $\tau^{H_i}$ on both sides, one has 
$$
\prod_i \det\left(\frac{1}{|H_i|} \langle, \rangle_1 \mid \tau^{H_i}\right)^{n_i} = \prod_i \det\left(\frac{1}{|H_i|} \langle, \rangle_2 \mid \tau^{H_i}\right)^{n_i}.
$$
\end{proposition}

\begin{proof}
	Follow proof of \cite[Theorem 2.17]{tamroot} where this is proved for $\bC[G / \Theta]  = 0$.
\end{proof}

The above proposition does not hold if $\tau$ appears as a subrepresentation of $\bC[G/ \Theta]$. If we restrict to considering $\cK = \bQ$ and $\bQ$-valued pairings, then we get the following result on independence of pairings up to norms.  

\begin{proposition}\label{prop:indepQ}
Let $G$ be a finite group, $\Theta = \sum_i n_i H_i \in \B(G)$ a $K$-relation for $G$. Let $\tau$ be a $\bQ G$-representation, and $\langle, \rangle_1$, $\langle, \rangle_2$ two non-degenerate $\bQ$-bilinear $G$-invariant pairings on $\tau$ that are $\bQ$-valued. 
Computing the determinants with respect to the same basis of $\tau^{H_i}$ on both sides, one has 
$$
\prod_i \det\left(\frac{1}{|H_i|} \langle, \rangle_1 \mid \tau^{H_i}\right)^{n_i} \equiv \prod_i \det\left(\frac{1}{|H_i|} \langle, \rangle_2 \mid \tau^{H_i}\right)^{n_i} \mod N_{K / \bQ}(K^{\times}).
$$
\end{proposition}

\begin{proof}
	It is enough to prove this for a particular choice of basis of each $\tau^{H_i}$. Suppose firstly that $\tau = \alpha \oplus \beta$ with $\alpha, \beta$ $\bQ G$-representations and $\Hom_G(\alpha, \beta^{\vee}) = 0$, where $\beta^{\vee}$ denotes the dual representation of $\beta$. Then $\alpha$ and $\beta$ are orthogonal with respect to $\langle, \rangle_1$, $\langle, \rangle_2$. Since $\tau^{H} = \alpha^H \oplus \beta^H$ for all $H \leq G$, choosing bases of $\tau^{H_i}$ that respects this decomposition reduces the problem to $\alpha$ and $\beta$ separately. Thus we may assume that $\tau = \rho^{\oplus n}$ where $\rho$ is an irreducible $\bQ G$-representation. As $\rho$ is irreducible it is of the form $\rho =  N_{\bQ(\chi) / \bQ}(\chi)^{\oplus m_{\chi}}$ where $\chi$ is a complex irreducible representation of $G$ and $m_{\chi}$ is the Schur index of $\chi$. 

Let $\langle, \rangle$ be a non-degenerate $G$-invariant $\bQ$-valued pairing on $\rho$. Then any other such pairing is of the form $\langle, \rangle_{\alpha}$ for $\alpha \in \End_G(\rho)$ with $\langle x, y \rangle_{\alpha} = \langle \alpha(x), y \rangle$ for $x$, $y \in \rho$. 
Similarly the non-degenerate $G$-invariant $\bQ$-valued pairings on $\tau$ are in one-to-one correspondence with $\End_G(\tau)$. Note that 
$$
\End_G(\tau) = \End_G(\rho^{\oplus n }) \simeq M_n(\End_G(\rho)) \simeq M_n(D),
$$
where $D = \End_G(\rho)$ is a finite-dimensional division algebra over $\bQ$ with centre $\bQ(\chi)$ (cf. \cite[\S12.2]{Serre}).

Fix a subgroup $H \leq G$. 
Viewing $\rho^H$ as a $D$-vector space with basis $\{v_1, \ldots v_k \}$, $k = \dim_D \rho^H$, an endomorphism $\alpha \in D$ acts by the matrix $\alpha I_k$. The algebra $D$ has dimension $m_{\chi}^2$ over $\bQ(\chi)$, and there exists a splitting field $\bQ(\chi) \subset L \subset D$ such that $[D \colon L] = m_{\chi}$ and $D \otimes_{\bQ(\chi)} L \simeq M_{m_{\chi}}(L)$. Let $\{f_1, \ldots f_{m_{\chi}} \}$ be a basis for $D$ over $L$. With respect to the basis $\{ f_i v_j \}$ of $\rho^H$, $\alpha$ acts by a block diagonal matrix consisting of $k$ blocks corresponding to multiplication by $\alpha$ on $D$ as an $L$-vector space. If we further consider a basis $\{ e_1, \ldots e_t\}$ of $L$ as a $\bQ$-vector space, then $ \cB = \{ e_i f_j v_k \}$ is a $\bQ$-basis of $\rho^{H}$. Induce this to a basis of $(\rho^{H})^{\oplus n}$, and
let $C$ be the matrix of $\langle, \rangle$ on $\rho^H$ with respect to $\cB$. The matrix of $\langle, \rangle_1$ on $\tau^H$ is of the form 
\[
\left( 
\begin{smallmatrix} 
C & 0 & \cdots & 0 \\ 
0 & \ddots & \ddots & \vdots \\  
\vdots & \ddots & \ddots & 0 \\ 
0 & \cdots & 0 & C
\end{smallmatrix} 
\right)
\left( 
\begin{smallmatrix} 
A_{11} & A_{12} & \cdots & A_{1n} \\ 
A_{21} & \ddots & \ddots & \vdots \\ 
\vdots & \ddots & \ddots & \vdots \\ 
A_{n1} & \cdots & \cdots & A_{nn} 
\end{smallmatrix} 
\right), 
\quad 
A_{ij} = 
\left( 
\begin{smallmatrix} 
\alpha_{ij} & 0 & \cdots & 0 \\ 
0 & \alpha_{ij} & 0 & \vdots \\ 
\vdots & \ddots & \ddots & 0 \\ 
0 & \cdots & 0 & \alpha_{ij}  
\end{smallmatrix} 
\right)
\]
with each $A_{i j} \in M_{k \cdot\dim_{\bQ}D}(\bQ)$ a block diagonal matrix consisting of $k$-blocks $\alpha_{i j}$, where $\alpha_{i j}$ is the matrix of multiplication by an element of $D$ with respect to the $\bQ$-basis  $\{ e_i f_j \}$ of $D$. In particular, the matrix of $\frac{1}{|H|} \langle, \rangle_1$ has determinant
$$
\det\left(\tfrac{1}{|H|} \langle, \rangle_1 \mid \tau^{H} \right) = \det\left(\tfrac{1}{|H|} C \right)^n \cdot  \det(\Lambda_{\alpha})^{\dim_D \rho^H},
$$
where $\Lambda_{\alpha}$ is the block matrix $\Lambda_{\alpha} = (\alpha_{i j })_{i, j}$ which does not depend on $H$. 

The determinant of $\Lambda_{\alpha}$ is the norm of an element of $\bQ(\chi)^{\times}$. To see this,
view $\Lambda_{\alpha}$ as a block matrix  $(m_{x_{i j }})_{i, j}$, where $m_{x_{i, j}}$ corresponds to the multiplication by $x_{i,j} \in L$ map on $L$ as a $\bQ$-vector space with basis $\{e_i\}$. Let $R = \{ m_y \colon y \in L\}$. This is a commutative subring of $M_{[L \colon \bQ]}(\bQ)$. We can view $\Lambda_{\alpha}$ as a matrix in $M_n(R)$, and consider the determinant $\det_R(\Lambda_{\alpha}) \in R$ as $R$ is commutative. By \cite[Theorem 1]{blockmatrices},
$$
\det(\Lambda_{\alpha}) = \det(\det\nolimits_R(\Lambda_{\alpha})) = \det(m_z) = N_{L / \bQ}(z), \quad z = \det(x_{i,j})_{i,j} \in L,
$$
hence $\det(\Lambda_{\alpha}) \in N_{\bQ(\chi) / \bQ}(\bQ(\chi)^{\times})$. 

The matrix of $\langle , \rangle_2$ with respect to our chosen basis is of the same form and  
$$
\det\left(\tfrac{1}{|H|} \langle , \rangle_2 \mid \tau^{H}\right) = \det\left(\tfrac{1}{|H|} C\right)^n \cdot \det(\Lambda_{\beta})^{\dim_D \rho^H}
$$
with $\Lambda_{\beta}$ not depending on $H$ and $\det(\Lambda_{\beta}) \in N_{\bQ(\chi) / \bQ}(\bQ(\chi)^{\times})$. 
Note that
$$
 \sum_i n_i \dim_D \rho^{H_i} = \tfrac{1}{[D \colon \bQ]}\sum_i n_i \dim_{\bQ} \rho^{H_i} = \tfrac{1}{[D \colon \bQ]} \sum_i n_i \langle \Ind_{H_i}^G \trivial, \rho \rangle  
 $$
 $$
 = \tfrac{1}{[D \colon \bQ]} \langle \bC[G / \Theta] , \rho \rangle = \langle \bC[G / \Theta], \chi \rangle.
$$
Taking the product of the ratio of the determinants over all $H_i$ we have
$$
\prod_i \big{(}\det(\tfrac{1}{|H_i|} \langle, \rangle_1 \mid \tau^{H_i})^{n_i} / \det(\tfrac{1}{|H_i|} \langle, \rangle_2 \mid \tau^{H_i})^{n_i}\big{)} = N_{\bQ(\chi) / \bQ}(s)^{\langle \bC[G / \Theta], \chi \rangle}
$$ 
for some $s \in \bQ(\chi)$. The result follows as by Proposition \ref{prop:norm_rel_irrs}, $\langle \bC[G / \Theta], \chi \rangle$ is divisible by $[K \colon K \cap \bQ(\chi)]$, and by Lemma \ref{lem:powernorm} $N_{\bQ(\chi) / \bQ}(s)^{[K \colon K \cap \bQ(\chi)]} \in N_{K / \bQ}(K^{\times})$.  
\end{proof}

Changing the basis of the subspace $\tau^{H_i}$ alters the determinant computation by a rational square, and so we get the following corollary.  

\begin{corollary}
If $\Theta$ is a $K$-relation for $G$ and $\tau$ a $\bQ G$-representation, the value 
$$
\cC^{\langle, \rangle}_{\Theta}(\tau) \in \bQ^{\times} /N_{K / \bQ}(K^{\times})\cdot \bQ^{\times 2}
$$ 
is independent of the choice of $\bQ$-valued pairing $\langle, \rangle$ and $\bQ$-bases for each invariant subspace $\tau^{H_i}$. In this case, we let $\cC_{\Theta}(\tau) \in \bQ^{\times}$ be a representative for the class of $\cC^{\langle, \rangle}_{\Theta}(\tau)$. 
\end{corollary}

\begin{corollary}[Regulator constants are multiplicative]\label{corr_reg_mult}
	Let $\Theta$, $\Psi \in \B(G)$ be $K$-relations for $G$, and let $\tau$, $\rho$ be self-dual $\cK G$-representations.
	\begin{enumerate}
		\item If  $\langle \bC[G / \Theta], \tau \rangle = \langle \bC[G / \Psi] , \tau \rangle = \langle \bC[G / \Theta], \rho \rangle = \langle \bC[G / \Psi] , \rho \rangle = 0$, then 
$$
\cC_{\Theta + \Psi}(\tau) \equiv \cC_{\Theta}(\tau) \cC_{\Psi}(\tau), 
\quad \cC_{\Theta}(\tau \oplus \rho) \equiv \cC_{\Theta}(\tau) \cC_{\Theta}(\rho) \quad \mod \cK^{\times 2}.
$$
\item If $\cK = \bQ$ and the regulator constants are computed with $\bQ$-valued pairings, then
$$
\cC_{\Theta + \Psi}(\tau) \equiv \cC_{\Theta}(\tau) \cC_{\Psi}(\tau), 
\quad \cC_{\Theta}(\tau \oplus \rho) \equiv \cC_{\Theta}(\tau) \cC_{\Theta}(\rho) \quad \mod N_{K / \bQ}(K^{\times})\cdot \bQ^{\times 2}.
$$
	\end{enumerate}
\end{corollary}

\begin{example}[{Permutation representations, \cite[Example 2.19]{tamroot}}]\label{ex:perm_pair} Consider $\tau = \bQ[G / D]$ for some subgroup $D \leq G$. By considering the standard ($\bQ$-valued) pairing on $\tau$, the elements of $G /D$ become an orthonormal basis. The space of invariants $\rho^H$ for $H \leq G$ has a basis consisting of $H$-orbit sums of these basis vectors. Then $\det(\langle, \rangle \mid \tau^H)$ is a product of the lengths of these orbits, and 
$$
\det\left(\tfrac{1}{|H|} \langle, \rangle \mid \tau^H \right) = \prod_{w \in H \backslash G / D} \tfrac{1}{|H|} \cdot \tfrac{|H|}{|\Stab_H(wD)|} = \prod_{w \in H \backslash G / D} \frac{1}{|H \cap D^{w}|}.
$$
By Corollary \ref{corr_reg_mult}(2), this computation can be used to compute $\cC_{\Theta}(\rho)$ for any rational representation $\rho$ that lies in the virtual permutation ring of $G$ and any $K$-relation $\Theta$ for $G$. 
\end{example}

\begin{proposition}\label{prop:not_self_dual}
	Let $\tau$ be a $\bQ G$-representation. Suppose that either 
	\begin{enumerate}
	\item all complex irreducible constituents of $\tau \otimes \bC$ are not self-dual, or
	\item $\tau \otimes \bC$ is symplectic.
	\end{enumerate}
	 Then $\cC_{\Theta}(\tau) \in N_{K / \bQ}(K^{\times})\cdot \bQ^{\times 2}$ whenever $\Theta$ is a $K$-relation for $G$.
\end{proposition}

\begin{proof}
	This is due to $\tau \otimes \bC$ admitting a non-degenerate alternating $G$-invariant pairing, so that in an appropriate basis the relevant determinant is a square (see \cite[Theorem 2.24, Corollary 2.25]{tamroot}).
\end{proof}

To compute regulator constants, it will be useful to view them as arising from a function on 
$\B(G)$.
With this in mind, we introduce the following definition and, from now on, restrict to the case $\cK = \bQ$. 

\begin{definition}
Let $\tau$ be a $\bQ G$-representation with a non-degenerate $\bQ$-valued $G$-invariant bilinear pairing $\langle, \rangle$. Define
$$ \cD_{\tau} \colon H \mapsto \det\left(\tfrac{1}{|H|} \langle, \rangle \mid \tau^{H} \right) \in \bQ^{\times} / \bQ^{\times 2},$$
where the determinant is taken with respect to any rational basis of $\tau^{H}$.
\end{definition}

Note that this depends on the choice of rational pairing. This defines a function on $\B(G)$ since by $G$-invariance $\cD_{\tau}(H) = \cD_{\tau}(xHx^{-1})$. Let $\Theta$ be a $K$-relation for $G$. If $\cD'_{\tau}$ is defined as above but with respect to a different rational pairing $\langle, \rangle'$, then by Proposition \ref{prop:indepQ} one has $\cD_{\tau}(\Theta) \equiv \cD'_{\tau}(\Theta) \mod N_{K / \bQ}(K^{\times})$ when computed with respect to the same basis for each $\tau^{H}$. Thus the value of $\cD_{\tau}(\Theta)$ in $\bQ^{\times} / N_{K / \bQ}(K^{\times}) \cdot \bQ^{\times 2}$ is independent of choice of basis and rational pairing, and   
$$\cD_{\tau}(\Theta) = \cC_{\Theta}(\tau) $$
by definition. 

\begin{example}
	Using the bases and pairings defined in Example \ref{ex:perm_pair}, $\cD_{\trivial}$ is the function $H \mapsto \tfrac{1}{|H|}$ and $\cD_{\bQ[G]}$ is the constant function $H \mapsto 1$. 
\end{example}

We are interested in computing $\cD_{\tau}$ on $K$-relations. If $K$ is quadratic, then we will consider $\cD_{\tau}$ as a function to $\bQ^{\times}$ (as opposed to $\bQ^{\times} / \bQ^{\times 2}$), since $\cD_{\tau} \sim_K \cD_{\tau}'$ when $\cD_{\tau}'$ is computed with respect to a different choice of basis or rational pairing. Note that in this case, $\cD_{\tau \oplus \tau'} \sim_K \cD_{\tau} \cD_{\tau'}$ where $\tau'$ is another $\bQ G$-representation. The following results focus on $K / \bQ$ quadratic. 

\begin{proposition}[Induction]\label{prop:induction}
If $D \leq G$ and $\tau$ is a $\bQ D$-representation, then 
$$ \cD_{\Ind_D^{G} \tau}(\Theta) \equiv (D, \cD_{\tau})(\Theta) \mod \bQ^{\times} / N_{K / \bQ}(K^{\times})\cdot \bQ^{\times 2}, $$ 
for a $K$-relation $\Theta$. In particular if $K / \bQ$ is quadratic, then $\cD_{\Ind_D^G \tau} \sim_K (D , \cD_{\tau})$. 
\end{proposition}

\begin{proof}
	The proof of the analogous statement for Brauer relations (\cite[Lemma 2.43]{tamroot}) works for $K$-relations also. 
\end{proof}

\begin{proposition}[{Induction and Restriction, \cite[Proposition 2.45]{tamroot}}]\label{prop:ind_res_reg}
Let $K / \bQ$ be quadratic, and $\tau$ a $\bQ G$-representation. 
\begin{enumerate}[(i)]
	\item If $G \leq G'$ and $\Theta \in \B(G')$ is a $K$-relation, then $\cC_{\Theta}(\Ind_{G}^{G'} \tau) \sim_K \cC_{\Res^{G'}_{G} \Theta}(\tau)$.
	\item If $D \leq G$ and $\Theta \in \B(D)$ is a $K$-relation, then $\cC_{\Theta}(\Res_D \tau) \sim_K \cC_{\Ind_D^G \Theta}(\tau)$.
\end{enumerate}
\end{proposition}

\begin{proof}
	(i) follows from Proposition \ref{prop:induction}, (ii) follows from the definition and noting that for $H \leq D$, $(\Res_D \tau)^{H} \simeq \tau^{H}$ as vector spaces. 
\end{proof}

\begin{corollary}[{\cite[Corollary 2.44]{tamroot}}]\label{cor:reg_perm}
	Let $K / \bQ$ be quadratic. Let $I \triangleleft D \leq G$ with $D / I$ cyclic. Then 
	$$ \cD_{\bQ[G / D]} \stackrel{\text{ Prop } \ref{prop:induction}}{\sim_K}(D, \cD_{\trivial}) \sim_K (D, H \mapsto \tfrac{1}{|H|}) \stackrel{\text{Cor }\ref{corr:cyclic_f_const}(i)}{\sim_K} (D, H \mapsto \tfrac{|D|}{|H|}) = (D, I, ef) . $$
\end{corollary}

\begin{lemma}\label{lem:cyclic_reg_const}
	Let $K / \bQ$ be quadratic, $\Theta$ a $K$-relation.
	\begin{enumerate}[(i)]
		\item If $D \leq G$ is cyclic, $\cC_{\Theta}(\bQ[G / D]) \in N_{K / \bQ}(K^{\times})$.
		\item If $G$ is cyclic then $\cC_{\Theta}(\tau) \in N_{K / \bQ}(K^{\times})$ for all rational representations $\tau$ of $G$.  
\end{enumerate}
\end{lemma}

\begin{proof}
By Corollary \ref{cor:reg_perm}, $\cD_{\bQ[G / D]} \sim_K (D, H \mapsto [D \colon H])\stackrel{\text{Lem } \ref{lem:f_on_cyclic}}{\sim_K} 1$, which proves $(i)$. For $(ii)$, since $G$ is cyclic,
every rational representation of $G$ can be written as a combination of permutation representations. Thus $\tau = \oplus_i \bQ[G / H_i]^{\oplus n_i}$ for $H_i \leq G$ and so $\cC_{\Theta}(\tau) \equiv \prod_i \cC_{\Theta}(\bQ[G / H_i])^{n_i} \equiv 1 \mod N_{K / \bQ}(K^{\times})$ by $(i)$ as each $H_i$ is cyclic.  
\end{proof}

\begin{lemma}\label{lem:odd_reg_const}
	Let $K / \bQ$ be quadratic, $\Theta$ a $K$-relation, and $G$ a group of odd order. Then $\cC_{\Theta}(\tau) \in N_{K / \bQ}(K^{\times})$ for all rational representations $\tau$ of $G$. 
\end{lemma}

\begin{proof}
	By multiplicativity of regulator constants, it is enough to show this for irreducible rational representations. If $\tau \in \Irr_{\bQ}(G)$ is non-trivial, then $\tau \otimes_{\bQ} \bC = \chi \oplus \overline \chi$ where $\chi$ is not self-dual. Thus $\cC_{\Theta}(\tau) \in N_{K / \bQ}(K^{\times})$ by Proposition \ref{prop:not_self_dual}.   
	Therefore it remains to show that $\cD_{\trivial} \sim_K 1$ and one has
	$$ \cD_{\trivial} \sim_K \cD_{\bQ[G]} \stackrel{\text{Cor \ref{cor:reg_perm}}}{\sim_K} 1. $$
\end{proof}

\section{Root numbers and Tamagawa numbers}\label{sec:main}

In this section we establish the following result, which expresses the product of Tamagawa numbers and other local fudge factors associated with a $K$-relation in terms of regulator constants and twisted root numbers. 
\begin{theorem}\label{corr:main}
Let $E$ be an elliptic curve over a number field $L$, with semistable reduction at the primes above $2$ and $3$ in $L$. Let $F / L$ be a finite Galois extension with $G = \Gal(F / L)$. If $\Theta = \sum_i n_i H_i$ is a $K$-relation for $G$ with $K / \bQ$ quadratic, then 
$$
\prod_i (C_{E / F^{H_i}})^{n_i} \equiv  \prod_{\tau \in \Irr_{\bQ}(G)}  \cC_{\Theta}(\tau)^{u(E  / L, \chi_{\tau})}
\mod N_{K / \bQ}(K^{\times}),
$$
where $\chi_{\tau}$ is a complex irreducible constituent of $\tau$, and  $u(E / L , \chi_{\tau}) \in \{0,1\}$ satisfies $w(E / L, \chi_{\tau}) = (-1)^{u(E / L, \chi_{\tau})}$ if $\chi_\tau$ is self-dual, and $u(E /L, \chi_\tau) = 0$ otherwise. 
\end{theorem}

We use this result to relate the norm relations test to twisted root numbers in $\S\ref{sec:applications}$. This theorem follows from the local result Theorem \ref{thm:main}, whose proof will take up the majority of this section. \\ 

Let $\cF / \cK / \bQ_l$ be finite extensions, with $\cF / \cK$ Galois and $D = \Gal(\cF / \cK)$. Let $E / \cK$ be an elliptic curve.
\begin{notation}\label{not:local_CD}
 For $H \leq \Gal(\cF / \cK)$ let 
$$ 
C_D(H) = C(E / \cF^{H}, \omega^0),
$$
as defined in \S\ref{sec:notation} for $\omega^0$ a differential on $E / \cK$, which we choose to be minimal. Extending $C_D$ additively defines a function on the Burnside ring of $D$. 
\end{notation}

We first show that for $K / \bQ$ a finite Galois extension, the evaluation of $C_D$ on $K$-relations for $D$ is independent of the choice of differential $\omega^0$.  

\begin{lemma}\label{lem:differential}
Let $K / \bQ$ be a finite Galois extension, and let $\Theta  = \sum_i n_i H_i$ be a $K$-relation for $D$. Then 
$$
\prod_i C(E / \cF^{H_i}, \omega^0)^{n_i} \equiv \prod_i C(E / \cF^{H_i}, \omega  )^{n_i} \mod N_{K / \bQ}(K^{\times})$$
for different choices of invariant differentials $\omega^0$, $\omega$. 
\end{lemma}

\begin{proof}
	Let $\alpha \in \cK^{\times}$ be such that $\omega = \alpha \omega^0$. Let $C_D$, $C_D'$ be the functions on $\B(D)$ as in Notation \ref{not:local_CD} defined using $\omega^0$ and $\omega$ respectively. Then 
$$
C_D'(\Theta) / C_D(\Theta) =\prod_i |\alpha|_{\cF^{H_i}} = (D, I, |\alpha|_{\cK}^{f})(\Theta).
$$
The function $(D, I, | \alpha|_{\cK}^{f} )$ passes to a quotient function $h$ on the Burnside ring of $D / I \simeq C_n$ given by $h(C_d) =  |\alpha|_{\cK}^{n / d} \in \bQ^{\times}$. For $d \mid n$, 
$$
h(\Psi_d) = |\alpha|_{\cK}^{\sum_{d' | d} \mu(d / d') d } = |\alpha|_{\cK}^{\Phi(d)} = |\alpha|_{\cK}^{[\bQ(\zeta_d) \colon \bQ]} \in N_{\bQ(\zeta_d) / \bQ}(\bQ(\zeta_d)^{\times}).  
$$
Hence by Example \ref{ex:cyclic_groups2} $h(\Theta\cdot I / I) \in N_{K / \bQ}(K^{\times})$ and the result follows. 
\end{proof}

 For the remainder of this section we restrict to $K$-relations with $K / \bQ$ a quadratic extension. Then, if $\Theta$ is a $K$-relation for $D$ and $\tau$ is a $\bQ D$-representation, $\cC_{\Theta}(\tau) 
 \mod N_{K / \bQ}(K^{\times})$ is independent of choice of $\bQ$-basis and $\bQ$-valued pairing on $\tau$. Our central result is as follows:

\begin{theorem}\label{thm:main}
	Let $\cF / \cK / \bQ_l$ be finite extensions, with $\cF / \cK$ Galois and $D = \Gal(\cF / \cK)$. Let $E$ be an elliptic curve over $\cK$. Assume either that $E$ has semistable reduction, or has additive reduction and $l \geq 5$.  Then there exists a $\bQ D$-module $\cV$ such that:
\begin{enumerate}
\item For all self-dual representations $\tau$ of $D$, $$\frac{w(E / \cK, \tau)}{w(\tau)^{2 }} = (-1)^{\langle \tau , \cV \rangle},$$ where $w(E / \cK, \tau)$ and $w(\tau)$ are local root numbers.
\item If $K / \bQ$ is a quadratic extension and $\Theta = \sum_i n_i H_i$ is a $K$-relation for $D$, then 
$$ 
\cC_{\Theta}(\cV) \equiv \prod_i C_D(H_i)^{n_i} \mod N_{K/\bQ}(K^{\times}),
$$
where the regulator constant $\cC_{\Theta}(\cV)$ is computed on any rational basis of $\cV$ with respect to any non-degenerate, $G$-invariant, $\bQ$-valued pairing on $\cV$.  
\end{enumerate}
\end{theorem}

\begin{remark}
	This theorem should be compared with \cite[Theorem 3.2]{tamroot}. We are proving the same compatibility result, but extending the second part of the theorem from Brauer relations to $K$-relations. 
\end{remark}

By taking the product over all finite places, this theorem enables the computation of the global product of fudge factors appearing in the BSD-quotient for a $K$-relation in terms of regulator constants, yielding Theorem \ref{corr:main}:   

\begin{proof}[Proof of Theorem \ref{corr:main}]
	Let $C \colon \B(G) \to \bQ^{\times}$ be the function given by $C(H) = C_{E / F^{H}}$. Then  
	$$
	C(\Theta) = \prod_i (C_{E / F^{H_i}})^{n_i} = \prod_v C_{D_v}(\Res_{D_v}\Theta),
$$
with the product ranging over all finite places $v$ of $L$ (cf. Example \ref{ex2:C}). Here $D_v \simeq \Gal(F_w / L_v) \leq G$ is the decomposition group for a choice of prime $w$ in $L$ above $v$. By Theorem \ref{thm:main} and Proposition \ref{prop:ind_res_reg}(i),
$$
C_{D_v}(\Res_{D_v}\Theta) \equiv \cC_{\Res_{D_v}\Theta}(\cV_v) \equiv \cC_{\Theta}(\Ind_{D_v}^G \cV_v)
\mod N_{K / \bQ}(K^{\times})
$$ 
for the $ \bQ D_v$-module $\cV_v$ satisfying Theorem \ref{thm:main}. Therefore 
$$
\prod_i (C_{E / F^{H_i}})^{n_i}  \equiv \prod_{\tau \in \Irr_{\bQ}(G)} \prod_v \cC_{\Theta}(\tau)^{\langle \chi_{\tau} , \Ind_{D_v} \cV_v \rangle} \mod N_{K / \bQ}(K^{\times}).
$$
By Proposition \ref{prop:not_self_dual}, if $\chi_{\tau}$ is not self-dual or is symplectic, then $\cC_{\Theta}(\tau) \in N_{K / \bQ}(K^{\times})$. Thus we only need to take the product over $\tau \in \Irr_{\bQ}(G)$ with $\chi_{\tau}$ orthogonal. In this case, $$(-1)^{\langle \chi_{\tau}, \Ind_{D_v} \cV_v \rangle} = \frac{w(E / L_v, \Res_{D_v} \chi_{\tau})}{w(\Res_{D_v} \chi_{\tau})^2}$$ by Frobenius reciprocity and Theorem \ref{thm:main}(1). As $\chi_{\tau}$ is orthogonal, $w(\chi_{\tau})^2 = 1$, and so $$(-1)^{\sum_v \langle \chi_{\tau}, \Ind_{D_v} \cV_v \rangle} = w(E / L, \chi_{\tau}) = (-1)^{u(E / L, \chi_{\tau})}.$$ Therefore $\sum_v \langle \chi_{\tau}, \Ind_{D_v} \cV_v \rangle \equiv u(E / L, \chi_{\tau}) \mod 2$ and computing mod $N_{K / \bQ}(K^{\times})$ one has 
$$
\prod_i (C_{E / F^{H_i}})^{n_i} \equiv \prod_{\substack{\tau \in \Irr_{\bQ}(G), \\ \tiny{\chi_{\tau} \text{ orthogonal}}}} \cC_{\Theta}(\tau)^{\sum_v \langle \chi_{\tau}, \Ind_{D_v} \cV_v \rangle} \equiv \prod_{\tau \in \Irr_{\bQ}(G)} \cC_{\Theta}(\tau)^{u(E / L, \chi_{\tau})}, 
$$
noting that $u(E / L, \chi_{\tau}) = 0$ when $\chi_{\tau}$ is symplectic since $w(E / L, \chi_{\tau}) = 1$ (see \cite[Lemma A.2(4)]{tamroot}).
\end{proof}

\subsection{Setup}

The proof of Theorem \ref{thm:main} is set up the same way as the proof of \cite[Theorem 3.2]{tamroot}.
We define the following extensions $\cL$ of $\cK$, depending on the reduction type of $E$:

\begin{notation}
\
\begin{enumerate}
	\item $E / \cK$ has semistable reduction.
		\begin{itemize}
			\item $(\text{1G})$, $E$ has good reduction, $\cL = \cK$, 
			\item $(\text{1S})$, $E$ has split multiplicative reduction, $\cL = \cK$, 
			\item $(\text{1NS})$, $E$ has non-split multiplicative reduction, $\cL$ is quadratic unramified. 
		\end{itemize}
	
	\item $E / \cK$ has additive reduction. In this case assume that $\cK$ has residue characteristic $l > 3$. Let $\Delta_{E}$ and $c_6$ 
		be the standard invariants for some model of $E / \cK$ and let $\fe = \frac{12}{\gcd(12, \ord \Delta_{E})}$. 
		\begin{itemize}
			\item $(\text{2C})$, $E$ has potentially good reduction, $\mu_{\fe} \subset \cK$, $\cL = \cK(\sqrt[\fe]{\Delta_{E}})				$ is a cyclic extension of $\cK$,
			\item $(\text{2D})$, $E$ has potentially good reduction, $\mu_{\fe} \not\subset \cK$, $\cL = \cK(\mu_{\fe}, 						\sqrt[\fe]{\Delta_{E}})$ is a dihedral extension of $\cK$,
			\item $(\text{2M})$, $E$ has potentially multiplicative reduction, $\cL = \cK(\sqrt{-c_6})$. 
		
		\end{itemize}
\end{enumerate}
\end{notation}

These are the minimal Galois extensions over which $E$ attains split semistable reduction. As in \cite[Hypothesis 3.6]{tamroot}, we claim that we may assume that $\cL$ is a subfield of $\cF$ when proving Theorem \ref{thm:main}. 

\begin{lemma}
Let $\cK \subset \cF \subset \cM$ with $\cM / \cK$ and $\cF / \cK$ Galois. 
Let $\cV$ be a $\bQ \Gal(\cM  / \cK)$-module. If $\cV$ satisfies the conditions of Theorem \ref{thm:main} for $E / \cK$ and the extension $\cM / \cK$, then $\cW = \cV^{\Gal(\cM / \cF)}$ satisfies the conditions of Theorem \ref{thm:main} for $E / \cK$ and the extension $\cF / \cK$.
\end{lemma}

\begin{proof}
	Let $G = \Gal(\cM / \cK)$, $N = \Gal(\cM / \cF)$. If $\tau$ is a self-dual representation of $G / N$ then we can inflate it to a representation of $G$ and one has $(-1)^{\langle  \tau, \cW \rangle_{G / N}} = (-1)^{\langle \tau, \cV \rangle_{G}} = w(E / \cK, \tau) / w(\tau)^2$ by assumption.  Let $K / \bQ$ be quadratic and let $\Theta$ be a $K$-relation for $G / N$. By Proposition \ref{prop:K_rels_properties}(5), the lift $\tilde{\Theta}$ of $\Theta$ to $\B(G)$ is a $K$-relation for $G$, and moreover $C_{G / N}(\Theta) = C_G(\tilde{\Theta})$.   
	The virtual representation $\cV \ominus \cW$ has $(\cV \ominus \cW)^{N} = 0$, so that $\cC_{\tilde{\Theta}}(\cV \ominus \cW) = 1$. Since $\cC_{\tilde{\Theta}}(\cV \ominus \cW) = \cC_{\tilde{\Theta}}(\cV) / \cC_{\tilde{\Theta}}(\cW)$ by Corollary \ref{corr_reg_mult}(2), we have $\cC_{\tilde{\Theta}}(\cV) = \cC_{\tilde{\Theta}}(\cW) = \cC_{\Theta}(\cW)$. The result follows since $C_G(\tilde{\Theta}) \equiv \cC_{\tilde{\Theta}}(\cV) \mod N_{K / \bQ}(K^{\times})$.
\end{proof}

It follows that if Theorem \ref{thm:main} holds for $E / \cK$ and the extension $\cF \cL / \cK$, then it also holds for $E / \cK$ and the extension $\cF / \cK$. Thus we may assume that $\cF / \cK$ contains $\cL$ as a subfield, and we make this assumption from now on. 

Let $D' = \Gal(\cF / \cL) \triangleleft D$. Our choice of $\cV$ in Theorem \ref{thm:main} is based on the following result of Rohrlich, which provides formulae for local twisted root numbers.

\begin{lemma}[{\cite{RohG}, Theorem 2 }]\label{lem:root_formula}
For a self-dual representation $\tau$ of $D$, 
$$
w(E / \cK, \tau) = w(\tau)^2 \lambda^{\dim \tau} (-1)^{\langle \tau, V \rangle}
$$
where $\lambda \in \{\pm 1\}$ and $V$ is the representation of $D / D'$ given in the table below.
Here $\trivial$ denotes the trivial character, $\chi$ the non-trivial character of order $2$, $\eta$ the unramified quadratic character, and $\sigma$ the unique faithful $2$-dimensional representation of $D / D'$. 
\begin{table}[H]
\begin{center}
	\caption{Root numbers.}
\begin{tabular}{| l | l |  l |  l | }
\hline
Case & $D / D'$ & $\lambda$ & $V$\\
\hline
$\mathrm{(1G)}$ & $C_1$ & $1$ & $0$\\
$\mathrm{(1S)}$ & $C_1$ & $1$ & $\trivial$\\
$\mathrm{(1NS)}$ & $C_2$ & $1$ & $\eta$\\
$\mathrm{(2C)}$ & $C_2, C_3, C_4, C_6$ & $\epsilon$ & $0$\\ 
$\mathrm{(2D)}$ & $D_3, D_4, D_6$ & $-\epsilon$ & $\trivial \oplus \eta \oplus \sigma $ \\
$\mathrm{(2M)}$ & $C_2$ & $w(\chi)^2$ & $\chi$\\
\hline
\end{tabular}
\end{center}
\end{table}
\end{lemma}
Given $V$ as in the lemma above, define
$
\cV = V \oplus \begin{cases} 0 & \text{if } \lambda = 1, \\ \bQ[D] & \text{if }\lambda = -1 \end{cases}.
$
Then
$$
\frac{w(E / \cK, \tau)}{w(\tau)^2}  = \lambda^{\langle \tau, \bQ[D]\rangle} (-1)^{\langle \tau, V \rangle} = (-1)^{\langle \tau, \cV \rangle},
$$
and so part $(1)$ of Theorem \ref{thm:main}, follows for this choice of $\cV$. 

Observe that if $K / \bQ$ is quadratic and $\Theta$ is a $K$-relation for $D$, then $\cC_{\Theta}(\bQ[D]) \in N_{K / \bQ}(K^{\times})$ by Corollary \ref{cor:reg_perm} and so $\cC_{\Theta}(\cV) \equiv \cC_{\Theta}(V) \mod N_{K / \bQ}(K^{\times})$. Therefore to prove Theorem \ref{thm:main}, it remains to prove the following.

\begin{proposition}\label{prop:to_prove}
Let $\Theta = \sum_i n_i H_i \in \B(D)$ be a $K$-relation for $D$ with $K / \bQ$ a quadratic extension. Then 
$$
C_D(\Theta) = \prod_i C(E / \cF^{H_i}, \omega^0)^{n_i} \equiv \cC_{\Theta}(V) \mod N_{K / \bQ}(K^{\times})
$$
where $V$ is the $\bQ D$-representation given in Lemma \ref{lem:root_formula} depending on the reduction of $E / \cK$. In other words, $\cD_{V} \sim_K C_D$ as functions on $\B(D)$.  
\end{proposition}

Let $I$, $W$ denote the inertia and wild inertia subgroups of $D$ respectively. We proceed by  proving Proposition \ref{prop:to_prove} case-by-case for each possible reduction type. 

\subsection{Semistable reduction}

\subsubsection*{Good reduction} If $E / \cK$ has good reduction then $C(E / \cF^{H}, \omega^0) = 1$ for all $H \leq D$ so $ 1 = C_D(\Theta) = \cC_{\Theta}(V)$ for $V = 0$. 

\subsubsection*{Split multiplicative reduction} 
Let $n = v_{\cK}(\Delta)$ so that $E / \cK$ has Type $\I_n$ reduction with $C(E / \cK, \omega^0) = n$. For semistable reduction, $\omega^0$ remains minimal in all extensions of $\cK$, $C(E / \cF^{H}, \omega^0) = c(E / \cF^{H})$ for all $H \leq D$. Then $E / \cF^{H}$ has split multiplicative reduction with $C(E / \cF^{H}, \omega) = e_{\cF^{H} / \cK} \cdot n$. Thus $C_D \colon \B(D) \to \bQ^{\times}$ is given by $C_D = (D, I, e n)$.  
For split multiplicative reduction, $V = \trivial$ and $\cC_{\Theta}(\trivial)  = \cD_{\trivial}(\Theta)$.
Then $$
\cD_{\trivial} \stackrel{\text{Cor }\ref{cor:reg_perm} }{\sim_K} (D, I, ef)
\stackrel{\text{\tiny{Cor \ref{corr:cyclic_f_const}(ii)}}}{\sim_K} (D, I, e)
 \stackrel{\text{\tiny{Cor \ref{corr:cyclic_f_const}(i)}}}{\sim_K} (D, I, en) = C_D.
$$
\subsubsection*{Non-split multiplicative reduction}
Once again, let $n = v_{\cK}(\Delta)$ and assume $E / \cK$ has non-split multiplicative reduction. By Tate's algorithm (\cite[\S IV.9]{AdvancedTopics}) for $H \leq G$, $e = e_{\cF^{H} / \cK}$, and $f = f_{\cF^{H}/ \cK}$, we have 
$$ C(E / \cF^{H}, \omega) = \begin{cases}
	2 & \text{if } 2 \mid en, 2 \nmid f, \\
	en & \text{if } 2 \mid f,\\
	1 & \text{else}. 
\end{cases} 
$$
In this case $V = \eta$ with $\bQ[D / D'] = \bQ[D / D] \oplus \eta$. Thus $$\cD_{\eta} \sim_K \cD_{\bQ[D / D']} / \cD_{\trivial} \stackrel{\text{Cor } \ref{cor:reg_perm}}{\sim_K} (D', I, ef) / (D, I, ef). $$
One can show that $\cD_{\eta} \sim_K C_D$ by following \cite[\S3.iii.3]{tamroot}, using Corollary \ref{cor:reg_perm} in place of \cite[Corollary 2.44]{tamroot}, Corollary \ref{corr:cyclic_f_const} in place of \cite[Theorem 2.36(f)]{tamroot}, and Proposition \ref{prop:rename_descent} = \cite[Theorem 2.36(r), (d)]{tamroot}. We also need to show that $\varphi_W \sim_K 1$ where 
$$\varphi_{W} \colon H \mapsto \begin{cases} 
	1 & \text{if } H \not= W, \\ 2 & \text{if } H = W.
\end{cases} $$ 
Let $\overline{W} = W / \Phi$ be the maximal exponent $2$ quotient of $W$. In \cite[\S3.iii.3]{tamroot} it is shown that $\varphi_W(H) = \varphi_{\overline{W}}(\overline{H})$ and that $\varphi_{\overline{W}}(\Psi) \in \bQ^{\times 2} \subset N_{K / \bQ}(K^{\times})$ whenever $\Psi \in \B(\overline{W})$ is a Brauer relation. For $\tau \in \Irr_{\bC}(\overline{W})$, $\bQ(\tau) \subset \bQ(\zeta_{2^k})$ for some $k \geq 1$ and so $\varphi_{\overline{W}}$ is trivial on $\bQ(\tau)$-relations as it takes values $1$ or $2$. Thus by Proposition \ref{prop:from_irr}(i), $\varphi_{\overline{W}} \sim_K 1$ so that $\varphi_W \sim_K 1$ also (note $\hat{C}(\overline{W}) = 1$ as it is abelian).   

\subsection{Additive reduction}

Now suppose that $E / \cK$ has additive reduction and that $\cK$ has residue characteristic $l \geq 5$.
Let $\delta = v_{\cK}(\Delta_E)$ be the valuation of the minimal discriminant of $E / \cK$, and $q$ the size of the residue field of $\cK$.  

In this case, the proof of Proposition \ref{prop:to_prove} is considerably longer. For readability, we have moved some calculations to appendices \ref{appendix1}, \ref{appendix2}. 

\begin{notation}\label{not:add_fns}
We break up the functions
$$
C_D = c \cdot \omega, \quad \cD_{V} = a \cdot d,
$$
with 
$$
\begin{array}{l l }
	c(H) = c(E / \cF^{H}), & a(H) = \det(\langle, \rangle \mid V^{H}), \\ \\
	\omega(H) = \left|\frac{\omega_{E / \cK}^0}{\omega_{E / \cF^{H}}^0} \right|_{\cF^{H}}, & d(H) = |H|^{-\dim V^{H}}.
\end{array}
$$
The rational pairing on $V$ is obtained by writing $V$ as a combination of permutation representations and using the pairings defined in Example \ref{ex:perm_pair}. These functions are well-defined on conjugacy classes and define functions on $\B(D)$ taking values in $\bQ^{\times} / \bQ^{\times 2}$. 
\end{notation}
We first note that we may reduce to the case where $f_{\cF / \cK}$ is a power of $2$.

\begin{proposition}
	\label{lem:f_2}
	For the case of additive reduction, it suffices to prove Theorem \ref{thm:main} when $f_{\cF / \cK}$ is a power of $2$. 	
\end{proposition}

\begin{proof}
	Follow the proof of \cite[Lemma 3.13]{tamroot}.
\end{proof}

\subsubsection{Tamagawa term}\label{sec:tam_add}
We show that $c \sim_K a$ when $E / \cK$ has potentially good reduction, and $c \sim_K d$ when $E / \cK$ has potentially multiplicative reduction. We will use the corresponding result for Brauer relations from \cite[\S3.iv]{tamroot}, namely that if $\Psi \in \B(D)$ is a Brauer relation, then 
$$ 
\begin{array}{l l }
	c(\Psi) \equiv a(\Psi) \mod \bQ^{\times 2} & \text{Case $2C / 2D$,}\\
	c(\Psi) \equiv d(\Psi) \mod \bQ^{\times 2} & \text{Case $2M$.}

\end{array}
$$

First suppose that $E / \cK$ has potentially good reduction. Letting $\fe = e_{\cL / \cK}$, we have $\delta \fe \equiv 0 \mod 12$ as $E / \cL$ has good reduction, and $\fe$ is the smallest integer for which this holds. The case of potentially good reduction breaks up into the following two cases: 
$$
\begin{array}{l  l  l l l }
	\text{Case $2C$:}  & D / D' \simeq C_{\fe} & V =  0 &(\fe \in \{ 2,3,4,6\}, & q \equiv 1 \mod \fe), \\
	\text{Case $2D$:}  & D / D' \simeq D_{2\fe} & V = \trivial \oplus \eta \oplus \sigma &  (\fe \in \{ 3, 4, 6\},&  q \equiv -1 \mod \fe),
\end{array}
$$
where $V$ is a representation of $D / D' = \Gal(\cL / \cK)$, and $\eta$ and $\sigma$ are the quadratic unramified character and unique faithful representation of $D / D'$ respectively. 
The congruence $q \equiv \pm 1 \mod \fe$ follows from \cite[Theorem 2]{RohG}.

Since $V$ is a representation of $D / D'$, the function $a$ is lifted from $D / D' \simeq \Gal(\cL / \cK)$.
As argued in \cite[Proposition 3.16]{tamroot}, the function $c$ is lifted from $\Gal(\cL^u / \cK)$, where $\cL^u$ is the maximal unramified extension of $\cL$ contained in $\cF$. By Proposition \ref{prop:res_quot}(ii) to show $c \sim_K a$ we may replace $\cF$ by $\cL^{u}$ and $D$ by $\Gal(\cL^{u} / \cK)$. Thus we have $e_{\cF / \cK} = \fe$, $f_{\cF / \cK} = 2^k$ with $k \geq 0$, and 
\begin{equation}\label{nice-D}
	D = C_{\fe} \rtimes C_{2^k} = \langle x, y \mid x^{\fe} = y^{2^k} = 1, y x y^{-1} = x^{\pm 1}\rangle,\tag{$\dagger$}
\end{equation}
with $x^{+1}$ in Case $2C$ and $x^{-1}$ in Case $2D$. 

\begin{proposition}\label{prop:tam_add}
	Let $E / \cK$ have additive reduction. If $E / \cK$ has potentially good reduction, then $c \sim_K a$. If $E / \cK$ has potentially multiplicative reduction, then $c \sim_K d$. 
\end{proposition}

\begin{proof}
	The computations are deferred to the appendix. In Propositions \ref{tam-2C} and \ref{tam-2D}, we prove that $c \sim_K a$ as functions on $\B(D)$ when $D$ is of the form \eqref{nice-D}. For the case of potentially multiplicative reduction, we establish that $c \sim_K d$ in Proposition \ref{prop:tam_d}.  
\end{proof}

\subsubsection{Differential term}

We next consider the function $\omega$. 
As described in \cite[\S3.iv]{tamroot} one has
$$
\omega = \begin{cases}
	\left(D, I, q^{\left\lfloor \delta e  / 12  \right\rfloor f}\right), & \text{if $E$ has potentially good reduction,} \\
	\left(D, I, q^{\left\lfloor e / 2 \right\rfloor f}\right), & \text{if $E$ has potentially multiplicative reduction.} 
\end{cases}
$$

We first focus on potentially good reduction and show that $\omega \sim_K d$. Recall that $q \equiv (-1)^{t} \mod \fe$, where $t = 0$ in Case $2C$ and $t = 1$ in Case $2D$. If $q$ is an even power of the residue characteristic $l$, then $t = 0$ and $\omega \sim_K 1 \sim_K d$ (since $V = 0$ in Case $2C$). Thus we may assume that $q$ is an odd power of $l$ and so $q \equiv l \mod \fe$.  

Following the proof of \cite[Proposition 3.18]{tamroot}, we have 
$$ 
\left( D, I, q^{\left\lfloor \delta e /12 \right\rfloor f} \right) = \left(I, W, l^{\left\lfloor \delta e f / 12 \right\rfloor} \right)
$$ 
as functions to $\bQ^{\times} / \bQ^{\times 2}$. As $W$ is an $l$-group, we define $k$ such that $e = l^k$, and define $n \in \{ 0, 1 \}$ by 
$$
l \delta \equiv q \delta \equiv (-1)^t\delta + 12n \mod 24. 
$$
Then
$$
\omega = \left(I, W, l^{\left\lfloor (-1)^{t k } \delta f / 12 \right\rfloor} \right)\left(I, W, e^{fn} \right)
$$
as a function to $\bQ^{\times}/ \bQ^{\times 2}$. We first show the second factor is trivial.

\begin{lemma}
One has $\left(I, W, e^{fn}\right) \sim_K 1$.
\end{lemma}

\begin{proof}
	If $n = 0$ this is trivial, so assume that $n = 1$. Then $(I, W, e^{f}) = (W, W, e)$. The group $W$ is nilpotent. Let $W_0 \triangleleft W$ be of index $l$. Then
	$$
	(W, W, e) = (W, W_0, ef) \stackrel{\text{\tiny{ Corr \ref{corr:cyclic_f_const}}}}{=} (W, W_0, e) = (W_0, W_0, e).
	$$
	By repeatedly taking index $l$ subgroups, it suffices to prove that $(C_l, C_l, e) \sim_K 1$, which is true by Lemma \ref{lem:f_on_cyclic}.
\end{proof}
If $t = 1$, then by \cite[Proposition 3.18]{tamroot},
$$
\left( I, W, l^{\left\lfloor (-1)^{tk} \delta f / 12 \right\rfloor} \right) = \left(I, W, l^{\left\lfloor \delta f / 12 \right\rfloor} \right) \left(I, W, \left\{ \begin{smallmatrix} e, & \text{if } \fe \nmid f \\ 1, & \text{if }\fe \mid f \end{smallmatrix} \right. \right), 
$$
with only the first factor when $t = 0$. 

In Case $2C$, $d \sim_K 1$. In Case $2D$, the function $d \colon H \mapsto |H|^{-\dim V^{H}}$ depends on $\cL \cap \cF^{H}$ as $V$ is a $D / D'$-module. Following the proof of \cite[Proposition 3.19]{tamroot}, we can rewrite $d$ as in the following lemma.  
\begin{lemma}[{\cite[Proposition 3.19]{tamroot}}]
	In Case $2D$, $d \sim_K \left(I, W, \left\{ \begin{smallmatrix} ef, & \text{if }\fe \nmid f \\ 1, & \text{if }\fe \mid f,\end{smallmatrix} \right. \right)$.  
\end{lemma}
Thus 
$$ d / \omega \sim_K \begin{cases}
	(I, W, l^{\left\lfloor  \delta f / 12 \right\rfloor}), & \text{Case $2C$}, \\
	(I, W, l^{\left\lfloor \delta f /12 \right\rfloor})\cdot\left(I, W, \left\{\begin{smallmatrix} f, & \text{if }\fe \nmid f, \\ 1 & \text{if }\fe \mid f\end{smallmatrix}\right.\right), & \text{Case 2D}. 
\end{cases}
$$
In both cases, $d / \omega$ is a function of $f$ and so descends to a function on the quotient $I / W \simeq C_r$. In Appendix \ref{appendix2} we show that these functions are trivial on $K$-relations and so obtain the following result. 

\begin{proposition}\label{prop:pot_good}
Let $E / \cK$ have potentially good reduction. Then $d \sim_K \omega$. 
\end{proposition}

In the case of potentially multiplicative reduction, we show that $\omega \sim_K a$. 

\begin{proposition}
Let $E / \cK$ have potentially multiplicative reduction. Then $\omega \sim_K a$. 
\end{proposition}

\begin{proof}
	As a function on $\B(D)$, $\omega = (D, I, q^{\left\lfloor e / 2 \right\rfloor f})$. One has $\omega \sim_K 1$, as follows from the case of $\fe = 2$, $\delta = 6$ in Proposition \ref{prop:case2C}. The function $a$ is lifted from $D / D' \simeq C_2$. Since $C_2$ has only rational representations and no Brauer relations, by Proposition \ref{prop:from_irr}(ii), $a \sim_K 1$. Thus $\omega \sim_K a \sim_K 1$.  
\end{proof}

It follows that $\cD_V \sim_K C_D$ when $E / \cK$ has additive reduction, completing the proof of Proposition \ref{prop:to_prove}.

\section{Applications}\label{sec:applications}

In this section, we explore the implications of Theorem \ref{corr:main}. We prove Theorem \ref{thm:weak_app} (Theorem \ref{thm:weak} from the introduction), which shows that the norm relations test is weaker than parity tests. Additionally, we show that the norm relations test follows from the parity conjecture for twists (Corollary \ref{corr:nrt_true}). We deduce that, under certain assumptions, the product of $\BSD$-quotients corresponding to a $K$-relation is the norm of an element in $K$ (Corollary \ref{corr:parity}). 

Let $F / L$ be a Galois extension of number fields with $G = \Gal(F / L)$ and let $E / L$ be an elliptic curve. Consider a $K$-relation $\Theta = \sum_i n_i H_i$ for $G$ with $K / \bQ$ quadratic. Recall that Theorem \ref{corr:main} says that 
$$\prod_i (C_{E / F^{H_i}})^{n_i} \mod N_{K / \bQ}(K^{\times})$$ can be computed as a product of regulator constants. 

\begin{remark}\label{rem:reg_consts}
The following summarises properties of regulator constants. 
For an arbitrary group $G$ and $K$-relation $\Theta = \sum_i n_i H_i$, we associate to a rational representation $\tau$ of $G$ its regulator constant 
$$\cC_{\Theta}(\tau) \in \bQ^{\times} / N_{K / \bQ}(K^{\times}).$$ 
This is independent of the choice of rational pairing on $\tau$. Moreover, if $\langle \bC[G / \Theta], \tau \rangle = 0$, then the regulator constant can be computed with respect to any pairing (not necessarily $\bQ$-valued) and is still independent of the choice of pairing. Regulator constants also satisfy the following properties:

\begin{enumerate}[(a)]
	\item $\cC_{\Theta}(\tau \oplus \tau') = \cC_{\Theta}(\tau)\cdot \cC_{\Theta}(\tau')$,
	\item $\cC_{\Theta}(\tau) = 1$ if $\tau = \chi \oplus \overline{\chi}$ with $\chi \not\simeq \overline{\chi}$,
	\item $\cC_{\Theta}(\tau) = 1$ if $\tau$ is symplectic,
	\item If $D$ is cyclic then $\cC_{\Theta}(\bQ[G / D]) = 1$,
	\item Let $G = \Gal(F / L)$ be the Galois group of an extensions of number fields, $E / L$ an elliptic curve, and $\tau = E(F) \otimes_{\bZ} \bQ$. If $\langle \tau, \bC[G / \Theta] \rangle = 0$, then $\cC_{\Theta}(\tau) \equiv \prod_i \Reg_{E / F^{H_i}} \mod N_{K / \bQ}(K^{\times})$. 
\end{enumerate}
\end{remark}

\begin{remark}\label{remark:root_numbers}
For the reader's convenience, we state some properties of twisted root numbers; proofs can be found in \cite[Proposition A.2]{tamroot}.  

Let $F / L$ be a Galois extension of number fields with $G = \Gal(F / L)$ and let $E / L$ be an elliptic curve. Let $\rho \colon G \to \GL_n(\bC)$ be an Artin representation. The twisted root number $w(E / L , \rho) \in \{ \pm 1 \}$ satisfies the following properties:  
\begin{enumerate}[(a)]
    \item (Multiplicativity) If $\rho = \rho_1 \oplus \rho_2$, then
   $ 
    w(E / L,\rho) = w(E / L,\rho_1) w(E / L,\rho_2).
   $ 
    
    \item (Induction) If $\rho = \Ind_H^G \tau$ for a subgroup $H \le G$, then
   $ 
    w(E / L,\rho) = w(E/F^H, \tau).
   $ 
   \item (Galois equivariance) If $\rho$ is self-dual and $\rho'$ is a Galois conjugate of $\rho$, then 
   $ 
    w(E / L,\rho) = w(E / L,\rho').
   $
\item (Trivial representation) If $\rho = \trivial$ is the trivial representation, then
	$ w(E / L, \trivial) = w(E / L).$ 
\end{enumerate}

\end{remark}

\subsection{Norm relations test}

The norm relations test, as discussed in the introduction and introduced in \cite{dok-wier-ev}, only considers elliptic curves over $\bQ$. We will now consider it for an arbitrary number field $L$.  

\begin{conjecture}[{\cite[cf. Theorem 33]{dok-wier-ev}}]\label{thm:dok-wier-ev}
	Let $E / L$ be an elliptic curve, $F / L$ a Galois extension with Galois group $G$, and $\rho$ an irreducible representation of $G$. Let
	$$ N_{\bQ(\rho) / \bQ}(\rho)^{\oplus m} = \bigg{(} \bigoplus_{\fg \in \Gal(\bQ(\rho) / \bQ)} \rho^{\fg} \bigg{)}^{\oplus m}
	= \bigg{(} \bigoplus_i \bC[G / H_i] \bigg{)} \ominus \bigg{(} \bigoplus_j \bC[G / H_j'] \bigg{)} 
$$
for some $m \in \bZ$ and $H_i$, $H_j' \leq G$. If either
$ \prod_i C_{E / F^{H_i}} \big{/} \prod_j C_{E / F^{H_j'}} $
is not a norm from a quadratic subfield $\bQ(\sqrt{D}) \subset \bQ(\rho)$ (when such a subfield exists), or is not a rational square when $m$ is even, then $E$ has a point of infinite order over $F$. 

\end{conjecture}

Note that the outcome of the norm relations test can change when a Brauer relation is added.  
In the set up of Conjecture \ref{thm:dok-wier-ev}, one has 
\[
	N_{\bQ(\rho)/\bQ}(\rho)^{\oplus m} = \bC[G/\Theta] \oplus \bC[G/\Psi],
\] 
where $\Theta = \sum_i H_i - \sum_j H_j$ and $\Psi \in \B(G)$ is an arbitrary Brauer relation. Thus computing $C(\Theta)$ and $C(\Theta + \Psi) = C(\Theta)C(\Psi)$ may give different outcomes.  


We now show that if Conjecture \ref{thm:dok-wier-ev} predicts positive rank of $E / F$, then $E / L$ has a twisted root number equal to minus one. 

\begin{theorem}\label{thm:weak_app}
Let $E / L$ be an elliptic curve with semistable reduction at the primes above $2$ and $3$, and let $F / L$ be a Galois extension. Suppose that Conjecture \ref{thm:dok-wier-ev} predicts that $\rk E / F > 0$. Then there exists an irreducible orthogonal representation $\sigma$ of $\Gal(F / L)$ such that $w(E / L, \sigma) = -1$. \end{theorem}

\begin{proof}
	Consider any $\bQ(\sqrt{D}) \subset \bQ(\rho)$. Assuming the notation of Conjecture \ref{thm:dok-wier-ev}, $\Theta = \sum_i H_i - \sum_j H_j'$ is a $\bQ(\sqrt{D})$-relation for $G = \Gal(F / L)$. By Theorem \ref{corr:main}
$$
\prod_i C_{E / F^{H_i}}\Big{/}{\prod_j C_{E / F^{H_j'}}} \equiv \prod_{\tau \in \Irr_{\bQ}(G)} \cC_{\Theta}(\tau)^{u(E / L, \chi_{\tau})} \mod N_{\bQ(\sqrt{D}) / \bQ}(\bQ(\sqrt{D})^{\times}), 
$$
where $\chi_{\tau}$ is a complex irreducible constituent of $\tau$, and $u(E / L, \chi_{\tau}) \in \{0, 1\}$ satisfies $w(E / L, \chi_{\tau}) = (-1)^{u(E / L, \chi_{\tau})}$ if $\chi_\tau$ is self-dual, and $u(E  / L, \chi_\tau) = 0$ otherwise. 

If $\prod_i C_{E / F^{H_i}} / \prod_j C_{E / F^{H_j'}}$ is not a norm in $\bQ(\sqrt{D})$, there exists $\tau \in \Irr_{\bQ}(G)$ such that $\cC_{\Theta}(\tau) \not= 1$ and $w(E / L, \chi_{\tau}) = -1$. By Proposition \ref{prop:not_self_dual}, $\chi_{\tau}$ is orthogonal.   

Now suppose $m$ is even. Then $\Theta$ is a $K$-relation for every quadratic field $K$, as  
$$ \bC[G / \Theta] \simeq N_{K / \bQ}(N_{\bQ(\rho) / \bQ}(\rho)^{\oplus m / 2}). $$
Let $n$ be the squarefree part of $\prod_i C_{E / F^{H_i}} / \prod_j C_{E / F^{H_j'}}$ and assume $n \not=1$. Write $n = p_1 \cdots p_r$ for $p_i$ distinct primes. Then there exists a quadratic field $K'$ such that each $p_i$ is inert in $K'$ and so $n \not\in N_{K' / \bQ}(K'^{\times})$. Thus as in the case when $m$ is odd, there exists a rational irreducible representation $\tau$ with orthogonal irreducible constituent $\chi_{\tau}$ such that $w(E / L, \chi_{\tau}) = -1$. 
\end{proof}

If one additionally assumes that the parity conjecture for twists holds, then $w(E / L, \sigma) = -1$ implies that $\sigma$ appears as a subrepresentation of $E(F)_{\bC}$ and so $\rk E / F > 0$. Thus we get the following. 

\begin{corollary}\label{corr:nrt_true}
	Let $E / L$ be an elliptic curve with semistable reduction at the primes above $2$ and $3$, and let $F / L$ be a Galois extension. Suppose that the parity conjecture for twists holds for all self-dual representations of $\Gal(F / L)$. Then Conjecture \ref{thm:dok-wier-ev} holds for $E / L$ and $F / L$.  
\end{corollary}

From the properties of regulator constants, we deduce some cases where the norm relations test cannot predict positive rank:  

\begin{corollary}\label{corr:fail}
Let $E / L$ be an elliptic curve, $F / L$ a Galois extension. Assuming $E / L$ has semistable reduction at the primes above $2$ and $3$, the norm relations test cannot predict that $\rk E / F > 0$ in the following cases:
\begin{enumerate}[(i)]
	\item $F / L$ is an odd order extension,
	\item $\Gal(F / L)$ is cyclic,
	\item $E / L$ has good reduction at the primes that ramify in $F$,
	\item Each prime $p$ at which $E$ has bad reduction has a cyclic or odd order decomposition group in $\Gal(F / L)$. 
\end{enumerate}
\end{corollary}

\begin{proof}
	Let $G = \Gal(F / L)$. Let $\rho$ be an irreducible representation of $G$, and $\Theta  = \sum_i n_i H_i \in \B(G)$ such that 
	$$ N_{\bQ(\rho) / \bQ}(\rho)^{\oplus m} =  \bC[G / \Theta].$$
	If $\prod_i (C_{E / F^{H_i}})^{n_i} \in N_{K / \bQ}(K^{\times})$ for all quadratic subfields $K \subset \bQ(\rho)$, or is a rational square when $m$ is even, then the norm relations test cannot predict that $\rk E / F > 0$. 

	$(i), (ii)$: Assume that $G$ is cyclic or of odd order.  By Lemma \ref{lem:cyclic_reg_const} and Lemma \ref{lem:odd_reg_const}, $\cC_{\Theta}(\tau) \in N_{K / \bQ}(K^{\times})$ for all rational representations $\tau$ of $G$ and quadratic extensions $K / \bQ$ for which $\Theta$ is a $K$-relation.
	Thus by Theorem \ref{corr:main}, $$\prod_i (C_{E / F^{H_i}})^{n_i} \in N_{K / \bQ}(K^{\times})$$ for all quadratic extensions $K / \bQ$ for which $\Theta$ is a $K$-relation. This includes all $K \subset \bQ(\rho)$, and when $m$ is even $\Theta$ is a $K$-relation for every quadratic extension $K / \bQ$, implying that this product is a rational square (cf. proof of Theorem \ref{thm:weak_app}).  

$(iii), (iv)$: Recall that 
$$ \prod_i (C_{E / F^{H_i}})^{n_i} = \prod_v C_{D_v}(\Res_{D_v} \Theta) $$
is a product of local terms for each finite place $v$, where $D_v$ is a decomposition group at $v$. By Theorem \ref{thm:main}, the value of $C_{D_v}(\Res_{D_v} \Theta) \mod N_{K / \bQ}(K^{\times})$ coincides with the regulator constant of a certain rational representation of $D_v$ for all $K / \bQ$ quadratic with $\Theta $ a $K$-relation. The only contributions are at primes of bad reduction, and in both cases the primes of bad reduction have cyclic or odd order decomposition groups. The result follows by arguing as for $(i)$, $(ii)$.  
\end{proof}

The previous result explains why dihedral groups are the simplest examples of when the norm relations test can be used to predict positive rank. 

\begin{example}[Dihedral group $D_{pq}$]\label{ex:dihedral}
Let $G = D_{pq}$ be the dihedral group of order $2pq$ for $p$, $q$ odd primes. The irreducible rational representations are
$\{\trivial$, $\varepsilon$, $\sigma_p$, $\sigma_q$, $\sigma_{pq} \}$, where $\trivial$ is the trivial representation, $\varepsilon$ the sign representation, $\sigma_p$, $\sigma_q$ factor through $D_{p}$, $D_q$ quotients respectively, and $\sigma_{pq}$ is faithful.  

Suppose that $p \equiv q \equiv 3 \mod 4$. For a faithful complex irreducible representation $\chi_{pq}$ of $D_{pq}$, $\bQ(\chi_{pq}) = \bQ(\zeta_{pq})^{+}$ is the maximal totally real subfield of $\bQ(\zeta_{pq})$, and $K = \bQ(\sqrt{pq}) \subset \bQ(\chi_{pq})$.
Consider the $K$-relation
$$
\Theta = C_2 - D_p - D_q + D_{pq}, \quad \bC[G / \Theta] = \sigma_{pq}. 
$$
Using Example \ref{ex:perm_pair} (=Remark \ref{rem:reg_consts}(d)), one computes
\begin{center}
\begin{tabular}{l l l l l }
	$\cC_{\Theta}(\trivial) = 1$, &
	$\cC_{\Theta}(\epsilon) = 1$, & 
	$\cC_{\Theta}(\sigma_p) =  q^{\frac{p - 1}{2}}$,& 
	$\cC_{\Theta}(\sigma_q) = p^{\frac{q - 1}{2}}$, &
	$\cC_{\Theta}(\sigma_{pq}) = p^{\frac{q - 1}{2}} q^{\frac{p - 1}{2}}$. 
\end{tabular}
\end{center}

Consider a Galois extension $F / L$ with $G = \Gal(F / L)$ and an elliptic curve $E / L$ that has semistable reduction at the primes above $2$ and $3$ in $L$.  
Let $\chi_{p}$, $\chi_{q}$ be complex irreducible constituents of $\sigma_{p}$ and $\sigma_{q}$ respectively. 
Then by Theorem \ref{corr:main}, 
$$C(\Theta) := \frac{C_{E / F^{C_2}} \cdot C_{E / L}}{C_{E / F^{D_p}} \cdot C_{E / F^{D_q}}} \equiv p^{a + c} \cdot q^{b + c} \mod N_{K / \bQ}(K^{\times})$$ where $a, b, c \in \{0,1\}$ satisfy $(-1)^a = w(E / L, \chi_{q})$, $(-1)^b = w(E / L, \chi_p)$, $(-1)^c = w(E / L, \chi_{pq})$.

Suppose that $p$ is a quadratic residue mod $q$ (then $q$ is not a quadratic residue mod $p$).
One has $p \in N_{K / \bQ}(K^{\times})$ but $q \not\in N_{K / \bQ}(K^{\times})$, and so 
$$ C(\Theta) \equiv q^{b + c} \mod N_{K / \bQ}(K^{\times}).$$
Thus if $C(\Theta) \not\in N_{K / \bQ}(K^{\times})$ it follows that $w(E / L, \chi_{p}) = -1$, or $w(E / L, \chi_{pq}) = -1$, but not both. 
\end{example}

\begin{example}[Twisted root number formula]
	Continuing the previous example, we can express $w(E / L, \chi_p)$ in terms of the root number of $E$ over intermediate subfields. Indeed, since $\bC[G / D_{q}] = \trivial \oplus \sigma_p$, using Remark \ref{remark:root_numbers} we compute
	$$
	\begin{array}{l l l }
		w(E / F^{D_q}) &  = w(E / L, \Ind_{D_q}^G \trivial) & \quad  \text{(Induction)}\\
			       &  = w(E / L) w(E / L, \sigma_p) & \quad \text{(Multiplicativity)}\\
			       & = w(E / L)w(E/ L, \chi_p)^{\frac{p-1}{2}} & \quad \text{(Galois equivariance)}\\
			       & = w(E / L) w(E / L, \chi_p) & \quad (p \equiv 3 \mod 4).  
	\end{array}
		$$
The twisted root number $w(E / L, \chi_q)$ can also be similarly expressed.  
On the other hand, it is not difficult to see that $w(E / L, \chi_{pq})$ cannot be computed as a product of $w(E / F^{H})$ for $H \leq G$. 
We get a curious formula for $w(E / L, \chi_{pq})$ in terms of the outcome of the norm relations test and root numbers of subfields:
$$ w(E / L, \chi_{pq}) = (-1)^{n} \cdot w(E / F^{D_q})\cdot w(E / L), \qquad n = \begin{cases}
	0 & \text{if } C(\Theta) \in N_{K / \bQ}(K^{\times}), \\ 1 & \text{if }C(\Theta) \not\in N_{K / \bQ}(K^{\times}).
\end{cases}$$
\end{example}

\subsection{BSD-quotients}

Our second application of Theorem \ref{corr:main} is in understanding the ratio of BSD-quotients corresponding to a $K$-relation. Recall that if $E / F$ is an elliptic curve over a number field $F$, its BSD-quotient is
$$ \BSD(E / F) = \frac{\Reg_{E / F} | \sha_{E / F}| C_{E / F}}{|E(F)_{\tors}|^2}, $$
(assuming that $\sha_{E  /F}$ is finite).
Since regulators of elliptic curves can be computed via regulator constants, we get the following compatibility between regulators and Tamagawa numbers.  

\begin{corollary}\label{corr:parity}
Let $E$ be an elliptic curve over a number field $L$, and $F / L$ a finite Galois extension with $G = \Gal(F / L)$. Suppose that the parity conjecture for twists holds for all self-dual representations $\tau$ of $G$. 

Let $\Theta = \sum_i n_i H_i$ be a $K$-relation for $G$ with $K / \bQ$ quadratic. If $\langle \bC[G / \Theta] , E(F)_{\bC} \rangle = 0$, then 
$$
\prod_i (C_{E / F^{H_i}})^{n_i} \equiv \prod_i (\Reg_{E / F^{H_i}})^{n_i} \mod N_{K / \bQ}(K^{\times}).
 $$
 If one further assumes that $\sha_{E / F}$ is finite 
, then 
 $$
 \prod_i \BSD(E / F^{H_i})^{n_i} \in N_{K / \bQ}(K^{\times}).
 $$
\end{corollary}

\begin{proof}
	For each $\bQ G$-irreducible representation $\tau$, let $\chi_{\tau}$ be a complex irreducible constituent.
	By Theorem \ref{corr:main} and parity for twists we have   
$$
\prod_i (C_{E / F^{H_i}})^{n_i}  \equiv \prod_{\tau \in \Irr_{\bQ}(G)} \cC_{\Theta}(\tau)^{\langle \chi_{\tau} , E(F)_{\bC} \rangle} 
\mod N_{K / \bQ}(K^{\times}).
$$
Consider $\tau \in \Irr_{\bQ}(G)$ such that $\langle \chi_{\tau}, E(F)_{\bC} \rangle > 0$. Then $\langle \tau, \bC[G / \Theta] \rangle = 0$ by assumption, and so $\cC_{\Theta}(\tau)$ is independent of choice of pairing by Proposition \ref{prop:indep}. The product of regulators $\prod_i ({\Reg_{E / F^{H_i}}})^{n_i}$ is $\cC_{\Theta}(E(F) \otimes_{\bZ} \bQ)$ computed with respect to the height pairing. But by independence of pairing it is equivalent to compute this with respect to a rationally valued pairing and so 
$$
\prod_i(\Reg_{E / F^{H_i}})^{n_i} \equiv \prod_{\tau} \cC_{\Theta}(\tau)^{\langle \chi_{\tau}, E(F)_{\bC} \rangle}
\equiv \prod_i (C_{E / F^{H_i}})^{n_i} \mod N_{K / \bQ}(K^{\times}).  
$$
For the second statement, note that if $\sha_{E / F}$ is finite, then $\sha_{E / F^{H}}$ is also finite for all $H \leq G$ by \cite[Remark 2.10]{annals}. The result follows by observing that all other terms in $\BSD(E / F^{H})$ for $H \leq G$ are squares. 
\end{proof}

The following result was proved in \cite[Theorem 35]{dok-wier-ev}, under the assumption that $E / \bQ$ is semistable. It provides an example of a kind of reverse implication  
 of Corollary \ref{corr:parity}. 
Specifically, let $E / L$ be an elliptic curve, and let $F / L$ be a Galois extension with $G = \Gal(F / L)$. Suppose that $\Theta = \sum_i n_i H_i \in \B(G)$ is a $K$-relation for some quadratic extension $K / \bQ$, and assume that the relation
$$ \prod_i (C_{E / F^{H_i}})^{n_i} \equiv \prod_i (\Reg_{E / F^{H_i}})^{n_i} \mod N_{K / \bQ}(K^{\times}) $$ 
holds (for example this follows from \cite[Conjecture 4]{dok-wier-ev} when $L = \bQ$). Then one can try use this relation to obtain information about the multiplicity of a subrepresentation of $E(F)_{\bC}$.  

In particular we are able to shorten the proof of \cite[Theorem 35]{dok-wier-ev}, and to extend the result to allow $E$ to have additive reduction away from $2$ and $3$. 

\begin{theorem}[{\cite[Theorem 35]{dok-wier-ev}}]
	Assume \cite[Conjecture 4]{dok-wier-ev} holds. Let $F / \bQ$ be a Galois extension with $\Gal(F / \bQ) = D_{pq}$ for $p$, $q$ primes with $p, q \equiv 3 \mod 4$, and let $\rho$ be a faithful irreducible Artin representation factoring through $\Gal(F / \bQ)$.  

	Then for every elliptic curve $E / \bQ$ with semistable reduction at $2$ and $3$, if $\ord_{s = 1} L(E, \rho, s)$ is odd, then $\langle \rho, E(F)_{\bC} \rangle > 0 $.  
\end{theorem}

\begin{proof}
	Suppose that $\langle \rho, E(F)_{\bC} \rangle = 0$. Let $\Theta$ be the $K$-relation given in Example \ref{ex:dihedral}, for $K = \bQ(\sqrt{pq}) \subset \bQ(\rho)$. Assume that $q$ is not a quadratic residue mod $p$ (interchanging the role of $p$ and $q$ if necessary). The assumption that Conjecture \cite[Conjecture 4]{dok-wier-ev} holds implies that
	$$ \frac{\Reg_{E / F^{C_2}} \cdot \Reg_{E / \bQ}}{\Reg_{E / F^{D_p}} \cdot \Reg_{E / F^{D_q}}} \equiv \frac{C_{E / F^{C_2}} \cdot C_{E / \bQ}}{C_{E / F^{D_p}} \cdot C_{E / F^{D_q}}} \equiv q^{b + c} \mod N_{K / \bQ}(K^{\times}),$$
	where the second congruence follows from Example \ref{ex:dihedral}, with $(-1)^b = w(E / \bQ, \chi_p)$, $(-1)^c = w(E / \bQ, \chi_{pq}) = w(E /\bQ, \rho)$. On the other hand, the product of regulators can also be computed in terms of regulator constants as in the proof of Corollary \ref{corr:parity}, yielding 
	$$ \frac{\Reg_{E / F^{C_2}} \cdot \Reg_{E / \bQ}}{\Reg_{E / F^{D_p}} \cdot \Reg_{E / F^{D_q}}} \equiv q^{\langle \chi_p, E(F)_{\bC} \rangle + \langle \rho, E(F)_{\bC} \rangle} \mod N_{K / \bQ}(K^{\times}).$$

	Conjecture \cite[Conjecture 4]{dok-wier-ev} also assumes that $|\sha_{E / F^{H}}| < \infty$ for all $H \leq \Gal(F / \bQ)$, and so the known cases of the parity conjecture (see \cite[Theorem 1.3]{tamroot}) for $E / \bQ$, $E / F^{D_q}$, imply that $w(E / \bQ, \chi_p) = (-1)^{\langle \chi_p, E(F)_{\bC} \rangle}$. Therefore we must have $c \equiv \langle \rho, E(F)_{\bC} \rangle = 0 \mod 2$. It follows that $w(E / \bQ, \rho) = 1$.  
	The $L$-function $L(E, \rho, s)$ admits an analytic continuation to $\bC$ and satisfies its functional equation (see proof of \cite[Theorem 35]{dok-wier-ev}). Thus $w(E / \bQ, \rho) = 1$ implies that $\ord_{s = 1} L(E, \rho, s)$ is even. 
\end{proof}

\begin{remark}
	In the above proof, the non-triviality of $\cC_{\Theta}(N_{\bQ(\rho)/ \bQ}(\rho)) \mod N_{K / \bQ}(K^{\times})$ is essential in relating the parity of $\langle \rho, E(F)_{\bC}\rangle $ to the twisted root number. Unfortunately, this doesn't extend to dihedral groups of more composite order. If $\rho$ is a faithful irreducible representation of $D_n$ with $n$ divisible by three or more primes, then it appears that $\cC_{\Theta}(N_{\bQ(\rho) / \bQ}(\rho)) \in \bQ^{\times 2}$ for all $K$-relations $\Theta$ for $D_n$ and quadratic extensions $K / \bQ$. 
\end{remark}

\appendix
\section{Tamagawa numbers}\label{appendix1}

To compute the Tamagawa number at a place of additive reduction, we use the following lemma stated in \cite[\S3.iv.4]{tamroot}, which is a consequence of Tate's algorithm. 

\begin{lemma}\label{lem_add_tam}\label{tamagawa-num}\cite[Lemma 3.22]{tamroot}
    Let $\cK' /\cK / \bQ_l$ be finite extensions and $l \geq 5$. Let
    $$E \colon  y^2 = x^3 + Ax + B, \qquad A, B \in \cK$$
    be an elliptic curve over $\cK$ with additive reduction, with $\Delta = -16(4A^3 + 27 B^2)$. Let $\delta=v_{\cK}(\Delta)$ be the valuation of the minimal discriminant, and $e$ the ramification index of $\cK'/\cK$.
   
    If $E / \cK$ has potentially good reduction, then 
        \[
        \begin{array}{l l l l}
            \gcd(\delta e, 12) = 2 & \implies & c(E / \cK') = 1, & \quad (\II, \II^*) \\
            \gcd(\delta e, 12) = 3 & \implies & c(E / \cK') = 2, & \quad (\III, \III^*) \\
	    \gcd(\delta e, 12) = 4 & \implies & c(E / \cK') = \begin{cases} 1, & \text{if }\sqrt{B} \notin \cK' 
	    \\ 3, & \text{if }\sqrt{B} \in \cK' \end{cases}, & \quad (\IV, \IV^*) \\
		    \gcd(\delta e, 12) = 6 & \implies & c(E / \cK') = \begin{cases} 2, & \text{if }\sqrt{\Delta} \notin \cK' 
		    \\ 1 \ \text{or} \ 4, & \text{if }\sqrt{\Delta} \in \cK' \end{cases}, & \quad (\I_0^*) \\
		    \gcd(\delta e, 12) = 12 & \implies & c(E / \cK') = 1, & \quad (\I_0)
        \end{array}
        \]
     
        If $E / \cK $ has potentially multiplicative reduction of type $\I_n^*$, then
        \[
            \begin{array}{l l l l}
		    2 \nmid e, 2 \nmid n & \implies & c(E / \cK') = \begin{cases} 2, & \text{if }\sqrt{B} \not\in \cK', \\ 4, & \text{if }\sqrt{B} \in \cK'. \end{cases} & \quad (\I_{en}^*) \\
		    2 \nmid e, 2 \mid n & \implies & c(E / \cK') = \begin{cases} 2 & \text{if }\sqrt{\Delta} \not\in \cK', \\ 4 & \text{if }\sqrt{\Delta} \in \cK' \end{cases} & \quad ({\I_{en}^*}) \\
                2 \mid e, \sqrt{-6 B} \not\in \cK' & \implies & c(E / \cK' ) = 2 & \quad (\I_{en}, \text{ non-split}) \\
                2 \mid e, \sqrt{-6 B} \in \cK' & \implies & c(E / \cK') = en & \quad (\I_{en}, \text{ split})
                \end{array} 
            \]
    \end{lemma}

\begin{lemma}\label{lem:3_tame}
Let $\cK / \bQ_l$ be a finite extension and $l \geq 5$. Let 
$$ 
E \colon y^2 = x^3 + A x + B, \qquad A, B \in \cK
$$
be an elliptic curve over $\cK$ with Type $\II$ or Type $\II^*$ additive reduction, i.e. the minimal discriminant has valuation $v_{\cK}(\Delta) = 2$ or $10$ respectively. 
If there exists a degree $3$ totally tamely ramified extension of $\cK$, then $\sqrt{\Delta} \in \cK$. 
\end{lemma}

\begin{proof}
	$E / \cK$ has $j$-invariant $j = A^3 / \Delta$. As $E / \cK$ has potentially good reduction, $v_{\cK}(j) \geq 0$ so $3v_{\cK}(A) \geq v_{\cK}(\Delta)$. This ensures in each case that $v_{\cK}(\Delta) = v_{\cK}(4A^3 + 27 B^2) = 2 v_{\cK}(B)$, hence $v_{\cK}(B) \in \{1, 5\}$. Let $\pi$ be a uniformiser of $\cK$. One has
$$
\frac{\Delta}{16 \pi^{v_{\cK}(\Delta)}} = \frac{-4A^3 - 27 B^2}{\pi^{v_{\cK}(\Delta)}} \equiv \frac{-3(3B)^2}{\pi^{2 v_{\cK}(B)}} \mod \pi. 
$$
Denote the residue field of $\cK$ by $\kappa$ and let $q$ be its order.
As there exists a degree $3$ tamely ramified extension of $\cK$, $q \equiv 1 \mod 3$. If $l \equiv 1 \mod 3$ then $-3$ is a square in $\bF_l$ and hence in $\kappa$ also. Else $l \equiv -1 \mod 3$, implying that $\kappa / \bF_l$ is an even degree extension and so $-3$ is a square in $\kappa$ also. It follows by Hensel's Lemma that the unit $\frac{\Delta}{16 \pi^{v_{\cK}(\Delta)}}$ is a square in $\cK$ and so $\Delta \in \cK^{\times 2}$ as required. 
\end{proof}

To prove Proposition \ref{prop:tam_add} we prove that $c \sim_K a$ as functions on $\B(D)$ where $a$ is defined in Notation \ref{not:add_fns}, $I = C_{\fe} \leq D$, and $D$ is given by 
$$
	D = C_{\fe} \rtimes C_{2^k} = \langle x, y \mid x^{\fe} = y^{2^k} = 1, y x y^{-1} = x^{\pm 1}\rangle,
$$
with $x^{+1}$ in Case $2C$ and $x^{-1}$ in Case $2D$. 
We will need to understand the structure of $\hat{C}(D) = \Char_{\bQ}(D) / \Perm(D)$ in the following cases.

\begin{lemma}\label{lem:repC4C2k}
	The group $D = C_4 \rtimes C_{2^k} = \langle x, y \mid x^4 = y^{2^k} = 1, yxy^{-1} = x^{-1} \rangle$ for $k \geq 2$ has $\hat{C}(D) = \bZ / 2 \bZ$, generated by the lift of the faithful $2$-dimensional irreducible representation of $Q_8$ (the quaternion group) to $D$.  
\end{lemma}

\begin{proof}
	Since $D$ is a $2$-group, every rational representation of $D$ lies in $\Perm(D)$ (\cite{Segal}). Thus any non-trivial generator of $\hat{C}(D)$ comes from a irreducible representation with non-trivial Schur index. 
	Let $X = C_4 \times C_{2^{k - 1}} \leq D$. The irreducible representations of $D$ are either $1$ dimensional, or $2$ dimensional and of the form $\Ind_{X} \chi$ where $\chi$ is a character of $X$ whose restriction to $C_4$ has order $4$.  Following \cite[\S8.3]{Serre}, the representations of dimension one have Schur index one. For the $2$-dimensional irreducible representations, there are two with rational character, lifted from the $D_{4}$ and $Q_8$ quotients of $D$. The former has order $1$ in $\hat{C}(D)$ and the latter has order $2$. The other irreducible $2$-dimensional representations $\chi$ of $D$ have $\bQ(i) \subset \bQ(\chi)$. Thus by 
	a result of Roquette \cite[Corollary 10.14]{isaacs} these $\chi$ have Schur index $1$. 
\end{proof}

\begin{lemma}\label{lem:repC6C2k}
	Let $D = C_6 \rtimes C_{2^k} = \langle x, y \mid x^6 = y^{2^k} = 1, y x y^{-1} = x^{-1} \rangle$ with $k \geq 2$. 
	Let $\tau_{S_3}$, $\tau_{D_{6}}$ be the lift of the faithful $2$-dimensional irreducible representation from the $S_3$-quotient and $D_{6}$-quotient respectively. Then for each $2$-dimensional irreducible representation $\tau$ of $D$ with $ \tau \not= \tau_{S_3}$, $\tau_{D_{6}}$, $N_{\bQ(\tau)  / \bQ}(\tau)$ has order $2$ in $\hat{C}(D)$.  
\end{lemma}

\begin{proof}
	The irreducible representations of $D$ are of dimension $1$ or $2$, and the characters do not contribute to $\hat{C}(D)$. Following \cite[\S8.3]{Serre}, the two-dimensional irreducible representations of $D$ are of the form $ \tau = \Ind_{X}^D \chi$ where $X = C_6 \times C_{2^{k - 1}}$ and $\chi$ is a character of $X$ such that $\chi|_{C_3}$ is non-trivial. When $\chi$ has order $3$, $\tau = \tau_{S_3}$, and when $\chi$ has kernel $C_{2^{k - 1}} = \langle y^2 \rangle$, $\tau =  \tau_{D_{6}}$. 

It is clear that $\tau_{S_3}$ and $\tau_{D_{6}}$ have order $1$ in $\hat{C}(G)$ as they are lifts of irreducible representations of order $1$ in $\hat{C}(S_3)$ and $\hat{C}(D_{6})$ respectively. We claim that $N_{\bQ(\tau) / \bQ}(\tau)$ has order $2$ in $\hat{C}(D)$ for all other $2$-dimensional irreducible representations $\tau$. 
Up to conjugacy, the only non-normal subgroups of $D$ are $C_{2^{k}}^a = \langle y \rangle$, $ C_{2^{k}}^b = \langle x^3 y \rangle$, and $C_2 \times C_{2^k} = \langle x^3, y \rangle$. One has 
$$ \Ind_{C_{2^{k}}^a} \trivial = \trivial \oplus \varepsilon_1 \oplus \tau_{S_3} \oplus \tau_{D_{6}}, \ \Ind_{C_{2^{k}}^b} \trivial = \trivial \oplus \varepsilon_2 \oplus \tau_{S_3} \oplus \tau_{D_{6}}, \ \Ind_{C_2 \times C_{2^k}} \trivial = \trivial \oplus \tau_{S_3}.$$
Any other subgroup $H \leq D$ is normal, and so if $\tau$ is a subrepresentation of $\Ind_H \trivial$, it occurs with multiplicity two (equal to its dimension). This shows that $N_{\bQ(\tau) / \bQ}(\tau)$ is not an element of $\Perm(D)$ and hence has order $>1$ in $\hat{C}(D)$.
If $\Theta_{\tau} \in 
\B(X)$ is such that $\bC[X / \Theta_{\tau}] = N_{\bQ(\chi) / \bQ} (\chi)$, then $\bC[D / \Theta_{\tau}] = (N_{\bQ(\tau) / \bQ} (\tau) )^{\oplus 2}$ as $[\bQ(\chi) \colon \bQ(\tau)] = 2$. As $X$ is abelian, such $\Theta_{\tau}$ exist for each $\tau$, and so the order of $\tau$ in $\hat{C}(D)$ is $2$.  
\end{proof}

\begin{proposition}\label{tam-2C}
In Case $2C$, one has $c \sim_K a$. 
\end{proposition}

\begin{proof}
As $a = 1$ we must show that $c \sim_K 1$. We establish this for each possible $\fe$.

\underline{Case $\fe = 2, 4$}. The exponent of $D = C_{\fe} \times C_{2^k}$ is a power of $2$, and $c(H) \in \{ 1, 2, 4 \}$ for $H \leq D$ by Lemma \ref{lem_add_tam}. Thus for each $\chi \in \Irr_{\bC}(D)$, $2$ is a norm from $\bQ(\chi)$ and so $c$ is trivial on $\bQ(\chi)$-relations. By \cite[Proposition 3.16]{tamroot}, $c(\Psi) \in \bQ^{\times 2} \subset N_{K / \bQ}(K^{\times})$ whenever $\Psi$ is a Brauer relation. Hence by Proposition \ref{prop:from_irr}(i), $c \sim_K 1$ (noting that $\hat{C}(D) = 1$ since it is abelian). 

\underline{Case $\fe = 3$.} In this case $D \simeq C_{3\cdot 2^k}$ is cyclic.
Using the notation of Lemma \ref{lem_add_tam}, first assume that $\sqrt{B} \in \cK$. One has 
	$$
	c = \left(D, H \mapsto \left\{ \begin{smallmatrix} 3 & & \text{if }3 \nmid [D \colon H] \\ 1 & & \text{if }3 \mid [D \colon H] \end{smallmatrix}\right. \right).
	$$
	By Lemma \ref{lem:alpha_beta}, $c \sim_K 1$. 
	Otherwise assume $\sqrt{B} \not\in \cK$. The condition $\delta = v_{\cK}(\Delta_E) \in \{4, 8\}$ ensures that $v_{\cK}(B)$ is even. Thus $\sqrt{B} \in \cK'$ whenever $f_{\cK' / \cK}$ is even. By Lemma \ref{lem_add_tam} 
	\[ c = \left( D, I, \left\{\begin{smallmatrix} 3 & & \text{if }e = 1, 2 \mid f \\ 1 & & \text{if } e = 1, 2 \nmid f \\ 1 & & \text{if } e = 3\phantom{, 2\mid f} \end{smallmatrix}\right.\right) \equiv  \left(D, H \mapsto \left\{\begin{smallmatrix} 3 & & \text{if } 6   \mid   [D : H] \\ 1 & &  \text{if }6  \nmid   [D : H ] \end{smallmatrix}\right.\right) \cdot \left(D, H \mapsto \left\{\begin{smallmatrix} 3 & & \text{if }2 \mid [D : H] \\ 1 & & \text{if }2 \nmid [D : H] \end{smallmatrix}\right. \right) ,\]
	as functions to $\bQ^{\times} / \bQ^{\times 2}$. Again by Lemma \ref{lem:alpha_beta} both functions on the right are trivial on $K$-relations and so $c \sim_K 1$.  

	\underline{Case $\fe = 6$.} 
	In this case there exists a degree $3$ tamely ramified extension of $\cK$, and so $\sqrt{\Delta} \in \cK$ by Lemma \ref{lem:3_tame}.  
	One has 
	\[ c(H) = \begin{cases} 1 & \text{if } e = 1, 6 ,\\ \left\{\begin{smallmatrix} 1 & &  \text{if } \sqrt{B} \not\in \cF^H \\ 3 & & \text{if } \sqrt{B} \in \cF^H \end{smallmatrix}\right. & \text{if } e = 2, \\ 1 \text{ or } 4 & \text{if }e = 3 .
         \end{cases} \] 
	 By \cite[Proposition 3.16]{tamroot}, $c(\Psi) \in \bQ^{\times 2}$ whenever $\Psi$ is a Brauer relation. Using Proposition \ref{prop:from_irr}(ii), to show that $c \sim_K 1$ it suffices to show that for each $\chi \in \Irr_{\bC}(D)$ there exists $\Theta_{\chi} \in \B(D)$ with $\bC[D / \Theta_{\chi}] = N_{\bQ(\chi) / \bQ}(\chi)$ and such that $c(\Theta_{\chi})$ is a norm from each quadratic subfield of $\bQ(\chi)$. 
	 The irreducible characters of $D$ are of the form $\chi \psi$ where $\chi$ is a character inflated from the $C_6$ quotient and $\psi$ is a character inflated from the $C_{2^k}$ quotient. The possible quadratic subfields of any $\bQ(\chi\psi)$ are $\bQ(i)$, $\bQ(\sqrt{\pm 2})$, $\bQ(\sqrt{\pm 3})$ and $\bQ(\sqrt{\pm 6})$.  
If $\bQ(\chi\psi)$ contains any of the quadratic subfields $\bQ(\sqrt{3})$, $\bQ(i)$, $\bQ(\sqrt{\pm 2})$, $\bQ(\sqrt{\pm 6})$, then $\psi$ has order $2^l$ where $l \geq 2$ (and $k \geq 2$). If $\chi$ has order $n$ then $\bC[D / \Theta_{\chi \psi}] = N_{\bQ(\chi\psi) / \bQ}(\chi\psi)$ where
$$
\Theta_{\chi\psi} = \sum_{d | n} \mu(d) (C_{6 / d} \times C_{2^{k - l}} - C_{6 / d} \times C_{2^{k - l + 1}}) \in \B(D). $$
One can check that $c(\Theta_{\chi\psi}) \in \bQ^{\times 2}$ in these cases, hence is a norm from every quadratic subfield of $\bQ(\chi\psi)$.
Otherwise if $\bQ(\chi\psi)$ contains $\bQ(\sqrt{-3})$ then $c(H) \in \{ \square, 3 \}$ for all $H \leq D$, and so $c(\Theta') \in N_{\bQ(\sqrt{-3} ) / \bQ}(\bQ(\sqrt{-3} )^{\times})$ for any $\Theta' \in \B(D)$ with $\bC[D / \Theta'] = N_{\bQ(\chi\psi) / \bQ}(\chi\psi)$.
\end{proof}

\begin{proposition}\label{tam-2D}
In Case $2D$, one has $c \sim_K a $.
\end{proposition}

\begin{proof}
Again we argue for each possible $\fe$.

\underline{Case $\fe = 4$.} The group $D = C_4 \rtimes C_{2^k}$ has $\hat{C}(D) = 1$ when $k = 0, 1$. When $k \geq 2$, then by Lemma \ref{lem:repC4C2k}  
$\hat{C}(D) = \bZ / 2\bZ$ is generated by the lift of the faithful irreducible representation of $Q_8$ to $D$. Denote this representation by $\chi$. Let $\Theta_{\chi}$ be the lift of $C_1 - C_2 \in \B(Q_8)$ to $\B(D)$. Then $\chi^{\oplus 2} = \bC[D / \Theta_{\chi}]$.
By Lemma \ref{lem_add_tam}, 
$c(\Theta_{\chi}) \equiv 1 \text{ or } 4 \mod  \bQ^{\times 2}$. The function $a$ is lifted from $D / D' \simeq D_4$, and $a(\Theta_{\chi}) = 1$ since $D' \cdot \Theta_{\chi} / D' = 0 \in \B(D / D')$. It follows that $a / c(\Theta_{\chi}) \in \bQ^{\times 2}$. 

By \cite[Proposition 3.16]{tamroot}, $a / c (\Psi) \in \bQ^{\times 2} \subset N_{K / \bQ}(K^{\times})$ whenever $\Psi \in \B(D)$ is a Brauer relation. 
The ratio $a / c(H)$ is a power of $2$ for all $H \leq D$, which is a norm form any quadratic subfield of $\bQ(\psi)$ for every irreducible representation $\psi$, as $D$ is a $2$-group. 

If $\Theta \in \B(D)$ is an arbitrary $K$-relation then $k = \langle \bC[D / \Theta], \chi \rangle$ is even and $\Theta' = \Theta - (k / 2)\Theta_{\chi}$ has $\langle \bC[D / \Theta'], \chi \rangle = 0$ and $a / c(\Theta') \equiv a / c(\Theta) \mod \bQ^{\times 2}$. Following the proof of Proposition \ref{prop:from_irr}(ii), one has $a/ c(\Theta') \in N_{K / \bQ}(K^{\times})$ and so $a / c \sim_K 1$.  

\underline{Case $\fe = 3$.} One has $D = C_3 \rtimes C_{2^k}$ with $k \geq 1$. If $\sqrt{B} \in \cK$, then 
$$ 
c(H) = \left\{\begin{smallmatrix} 3 & &  \text{if }3 \nmid e_{\cF^{H} / \cK} \\ 1 & & \text{if } 3 \mid e_{\cF^{H} / \cK} \end{smallmatrix}\right. = \left\{\begin{smallmatrix} 3 & & \text{if } 3 \nmid e_{\cF^{H} \cap L / \cK} \\ 1 & & \text{if }3 \mid e_{\cF^{H} \cap \cL / \cL} \end{smallmatrix}\right. = c'(D' H / D') 
$$
where $c'$ is the function on $D / D'$ given by $c'(H') = \left\{\begin{smallmatrix} 3 & & \text{if }3 \nmid [D / D' \colon H'] \\ 1 & & \text{if }3 \mid [D / D' \colon H'] \end{smallmatrix} \right.$ for $H' \leq D / D'$.  As the function $a$ is also lifted from $D / D' \simeq S_3$  by Proposition \ref{prop:res_quot}(ii) it suffices to prove that $a / c \sim_K 1 $ on $K$-relations in $S_3$. 
Up to multiples $S_3$ has a unique Brauer relation given by $\Psi = 2S_3 + C_1 - 2C_2 - C_3$ and one computes that $a / c'(\Psi) \in \bQ^{\times 2}$. Thus $a / c' \sim_K 1$ by Proposition \ref{prop:from_irr}(ii) since all irreducible representations of $S_3$ have rational character, and so $a \sim_K c$ as functions on $\B(D)$. If $\sqrt{B} \not\in \cK$, then as in Case $2C$ we have 
$$
c(H) = \left( D, H \mapsto \left\{\begin{smallmatrix} 3 & & \text{if } 3 \nmid e_{\cF^{H} / \cK},\ 2\mid f_{\cF^{H} / \cK} \\ 1 & & \text{otherwise}\phantom{sssssssssss} \end{smallmatrix}\right.\right) =  \left( D, H \mapsto \left\{\begin{smallmatrix} 3 & & \text{if } 3 \nmid e_{\cF^{H} \cap \cL / \cK},\ 2\mid f_{\cF^{H} \cap \cL / \cK} \\ 1 & & \text{otherwise} \phantom{ssssssssssssssss}\end{smallmatrix}\right.\right),
$$
so that $c(H) = c'(D' H / D')$ where $c'$ is the function on $D / D'$ given by $$c'(H') = \left\{ \begin{smallmatrix} 3 & & \text{if } 3 \nmid [D / D' \colon H'], \ 2 \mid[D / D' \colon H'] \\ 1 & & \text{otherwise} \phantom{ssssssssssssssssssss} \end{smallmatrix}\right..$$
Once more $a / c'(\Psi) \in \bQ^{\times 2}$ for $\Psi$ the unique Brauer relation of $S_3 \simeq D/ D'$ and so $a / c' \sim_K 1$ by Proposition \ref{prop:from_irr}(ii) and hence $a \sim_K c$ on $D$ by Proposition \ref{prop:res_quot}(ii).

\underline{Case $\fe  = 6$.}
Let $\tau_{S_3}$, $\tau_{D_{6}}$ be the $2$-dimensional irreducible representations of $D$ inflated from the $S_3$-quotient and $D_{6}$-quotient respectively.
Let $\tau$ be a $2$-dimensional irreducible representation of $D$ with $\tau \not= \tau_{S_3}$, $\tau_{D_{6}}$. 
Then $\tau$ is of the form $\tau = \Ind_{X}^D \chi\psi$ where $\chi\psi$ is a character of $X =  C_6 \times C_{2^{k - 1}}$ with $\chi$ of order $n \in \{3, 6\}$ inflated from the $C_6$-quotient and $\psi$ of order $2^l$, $l \geq 1$, inflated from the $C_{2^{k - 1}}$-quotient. The representation $N_{\bQ(\tau) / \bQ}(\tau)$ has order $2$ in $\hat{C}(D)$ by Lemma \ref{lem:repC6C2k} and $N_{\bQ(\tau) / \bQ}(\tau)^{\oplus 2} = \bC[D / \Theta_{\tau}]$ where 
$$ \Theta_{\tau} = \sum_{d \mid n} \mu(d) \left( C_{6/d} \times C_{2^{k - l}} - C_{6 / d} \times C_{2^{ k - l + 1}}  \right).$$
One computes that $a / c (\Theta_{\tau}) \in \{1, 4 \} \subset \bQ^{\times 2}$. 

The derived subgroup of $D$ is $[D, D] = C_3 = \langle x^2 \rangle$, and the one-dimensional representations are inflated from $D / [D, D] \simeq C_2 \times C_{2^k}$. If $\varphi$ is a one-dimensional representation and $\bQ(\varphi)$ contains a quadratic subfield, then $\varphi$ is of the form 
$\varphi = \chi\psi$ where $\chi$ is of order $n \in \{1, 2 \}$ inflated from the $C_2 \simeq D / \langle x^2, y \rangle$ quotient and $\psi$ is of order $2^l$ for $l \geq 2$ inflated from the $C_{2^k} \simeq D / \langle x \rangle$ quotient. One can take $\Theta_{\varphi} \in \B(D)$ to be the lift of 
$$ \sum_{d \mid n} \mu(d) (C_{2 / d} \times C_{2^{k - l}} - C_{2 / d} \times C_{2^{k - l + 1}}) \in \B(D / [D, D]). $$
In this case one computes that $a / c(\Theta_{\varphi}) = 1$  when $l \geq 2$ and hence is a norm from every quadratic subfield of $\bQ(\varphi)$.

Since $\tau_{S_3}$, $\tau_{D_6}$ are rational, they appear with even multiplicity in $\bC[D / \Theta]$. Moreover these have order one in $\hat{C}(D)$, and so $\tau_{S_3} = \bC[D / \Theta_{\tau_{S_3}}]$, $\tau_{D_6} = \bC[D / \Theta_{\tau_{D_6}}]$ for some $\Theta_{\tau_{S_3}}$, $\Theta_{\tau_{D_6}} \in \B(D)$. By \cite[Proposition 3.16]{tamroot}, $a / c(\Psi) \in \bQ^{\times 2}$ for any Brauer relation $\Psi \subset \B(D)$. Now let $\Theta$ be an arbitrary $K$-relation for $D$. By writing $\Theta$ as a sum of a Brauer relation and the previously defined $\Theta_{\tau}$, $\Theta_{\varphi}$, $\Theta_{\tau_{S_3}}$, $\Theta_{\tau_{D_{12}}}$ and arguing as in Case $\fe = 4$, one obtains $a / c (\Theta) \in N_{K / \bQ}(K^{\times})$ and so $a \sim_K c$.  
\end{proof}

\begin{proposition}\label{prop:tam_d}
	Let $\cF / \cK / \bQ_l$ be finite extensions with $l \geq 5$ and $\cF / \cK$ Galois with Galois group $D$.  Let $E / \cK$ be an elliptic curve with potentially multiplicative reduction. Assume that $f_{\cF / \cK}$ is a power of $2$, and that the quadratic extension $\cL$ of $\cK$ over which $E$ attains semistable reduction is a subfield of $\cF / \cK$.  

Then $c \sim_K d$ as functions on $\B(D)$, where $d$ is defined in Notation \ref{not:add_fns}, and $V$ is the quadratic character of the quotient $\Gal(\cL / \cK)$ of $D$.  
\end{proposition}

\begin{proof}
Let $n = v_{\cK}(\Delta)$ be the valuation of the minimal discriminant of $E / \cK$ and let $D' = \Gal(\cF / \cL)$. 
 Letting $e = e_{\cF^{H} / \cK}$, $f = f_{\cF^{H} / \cK}$, the values of $c / d$ are computed using Lemma \ref{lem_add_tam} as follows in \cite[Proposition 3.21]{tamroot}. 
	\begin{table}[H]
		\caption{Values of $c / d$.}
		\centering
		\begin{tabular}{l | l l l l}
			& $E / \cF^{H}$ & $c$ & $d$ & $c/ d$ \\
			\hline
			$\cL \subset \cF^{H}$ & $\I_{ne}$ split & $ne$ & $\tfrac{|D|}{ef}$ & $\tfrac{n}{|D|}e^2 f$ \\
			$\cL \not\subset \cF^{H}$ & $\I_{ne}$ non-split & $2$ & $1$ & $2$ \\
			$\cL \not\subset \cF^{H}$, $2 \mid e$, $f = 1$ & $\I_{ne}^*$ & $c(E / \cK)$ & $1$ & $c(E / \cK)$ \\
			$\cL \not\subset \cF^{H}$, $2 \mid e$, $2 \mid f$ & $\I_{ne}^*$ & $4$ & $1$ & $4$ 
		\end{tabular}
\end{table}
In the proof of \cite[Proposition 3.21]{tamroot} it is also shown that, as a function to $\bQ^{\times} / \bQ^{\times 2}$, $c / d$ factors through  
$$
\Gal(\cL^u / \cK) \simeq \Gal(\cL  / \cK) \times \Gal(\cK^u / \cK)  = C_2 \times C_{2^k} := A, 
$$
where $\cL^{u}$ and $\cK^{u}$ denote the maximal unramified extensions of $\cL$ and $\cK$ in $\cF$. 

As $A$ is an abelian $2$-group, if $K \not= \bQ(i)$, $\bQ(\sqrt{\pm 2})$ then $c / d \sim_K 1$ by Proposition \ref{prop:from_irr}(ii), since $c / d(\Psi) \in \bQ^{\times 2}$ whenever $\Psi$ is a Brauer relation by \cite[Proposition 3.21]{tamroot}, and no irreducible representation of $A$ contains $K$ in its field of values.

Thus we may assume that $K \in \{ \bQ(i), \bQ(\sqrt{\pm 2})\}$. Note that $c(E / \cK) \in \{2, 4\}$. Since $2$ is a norm from each of these subfields, to show that $c / d \sim_K 1$ it suffices to show that the term $n / |D|$ occurs evenly in $c / d(\Theta)$ whenever $\Theta$ is a $K$-relation for $D$.  
Let $\psi$ be the order two character of $D$ with kernel $D'$. Then 
$$
\langle \psi, \bC[D / H] \rangle = \begin{cases} 1 & \text{if } \cL \subset \cF^H, \\ 0 & \text{if }\cL \not\subset \cF^H. \end{cases}
$$
Thus $(c / d)(\Theta) \equiv (n/|D|)^{\langle \psi, \bC[D / \Theta]\rangle } \equiv 1 \mod N_{K / \bQ}(K^{\times}), 
$
since $\psi$ occurs with even multiplicity in $\bC[D / \Theta]$ by Proposition \ref{prop:norm_rel_irrs} as it is a rational character. 
\end{proof}

\section{Differential term}\label{appendix2}
We use the notation of \S\ref{sec:main}. To prove Proposition \ref{prop:pot_good}, we show that $d / \omega \sim_K 1$ as functions on $\B(D)$, where
$$ d / \omega \sim_K \begin{cases}
	(I, W, l^{\left\lfloor  \delta f / 12 \right\rfloor}), & \text{Case $2C$}, \\
	(I, W, l^{\left\lfloor \delta f /12 \right\rfloor})\cdot\left(I, W, \left\{\begin{smallmatrix} f, & \text{if } \fe \nmid f, \\ 1 & \text{if }\fe \mid f\end{smallmatrix}\right.\right), & \text{Case 2D}. 
\end{cases}
$$
Let $I / W \simeq C_r$. For each possible $\fe$, let $h_{\fe}$, $g_{\fe}$ be the functions  on $\B(C_r)$ given by 
$$ h_{\fe} \colon C_{r / n} \mapsto l^{\left\lfloor \delta n / 12 \right\rfloor}, \quad g_{\fe} \colon C_{r / n} \mapsto \left\{ \begin{smallmatrix} n, & \text{if }\fe \nmid n,  \\ 1 & \text{if }\fe \mid n \end{smallmatrix} \right. .$$ 
Let $\Theta \in \B(D)$ be a $K$-relation for $D$, so that $\Res_I \Theta \in \B(I)$ is a $K$-relation for $I$ by Proposition \ref{prop:K_rels_properties}(2). Let $W \cdot \Res_I \Theta / W = \sum_{n \mid r} a_n \Psi_n \in \B(C_r)$ where $\Psi_n$ are given in Example \ref{ex:cyclic_groups1}. Then
$$ 
\bC[I / \Res_I \Theta] = \bigoplus_{n \mid r} \chi_n^{\oplus a_n} \oplus A, 
$$
where $\chi_n$ is the rational representation inflated from $I / W$ with kernel $W \rtimes C_{r / n}$, and $A$ is a representation of $I$ with $A^{W} = 0$. Let $\psi_n$ be a linear character of order $n$ of $I$ such that $N_{\bQ(\zeta_n) / \bQ}(\psi_n) = \chi_n$. Then
$$ 
a_n = \langle \bC[I / \Res_I \Theta] , \psi_n \rangle = \langle \bC[D / \Theta], \Ind_{I}^D \psi_n \rangle, 
$$
and $a_n$ is even if $K \not\subset \bQ(\Ind_{I}^D \psi_n)$. Therefore
$$
d / \omega (\Theta) \equiv \begin{cases}
	\underset{n, K \subset \bQ(\Ind_I^{D} \psi_n)}{\prod} h_{\fe}(\Psi_n)^{a_n}  & \text{Case $2C$}, \\
	\underset{n, K \subset \bQ(\Ind_I^D \psi_n)}{\prod} (h_{\fe}\cdot g_{\fe})(\Psi_n)^{a_n} & \text{Case $2D$} 
\end{cases} 
\mod N_{K / \bQ}(K^{\times}).
$$

To prove that $d / \omega \sim_K 1$, we show that $h_{\fe}(\Psi_n)^{a_n}$, $(h_{\fe}\cdot g_{\fe})(\Psi_n)^{a_n} \in N_{K / \bQ}(K^{\times})$ for any quadratic subfield $K \subset \bQ(\Ind_I^D \psi_n)$. 
Note that $h_{\fe}(\Psi_n) \equiv l^a \mod \bQ^{\times 2}$ with $a \in \{0,1\}$. We first show that $l$ is the norm of an ideal of $K$.  

\begin{proposition}
	If $K \subset \bQ(\Ind_{I}^{D} \psi_n)$ then $l$ is the norm of an ideal of $K$. 
\end{proposition}
	
\begin{proof} We show that $\bQ(\Ind_I^{D} \psi_n) \subset \bQ(\zeta_n)^{\langle \zeta_n \mapsto \zeta_n^{q} \rangle}$, that is that $\bQ(\Ind_I^D \psi_n)$ lies in the subfield of $\bQ(\zeta_n)$ fixed by $\langle q \rangle \leq (\bZ / n \bZ)^{\times} \simeq \Gal(\bQ(\zeta_n) / \bQ)$. Let $H = \ker \psi_n \leq I$. This is a characteristic subgroup of $I$ and so normal in $D$ also. Thus $\Ind_I^D \psi_n$ is the lift of $\Ind_{I / H}^{D / H} \psi_n$, viewing $\psi_n$ as a representation of the quotient $I / H$. One has $D / H = \Gal(\cK' / \cK)$, where $\cK' / \cK$ is a tamely ramified extension with inertia group $I / H$. Let $\sigma \in D / H$ be a lift of Frobenius from $(D / H) / (I / H)$. Then for $\tau \in I / H$, $\sigma \tau \sigma^{-1} = \tau^q$. The character of $\Ind_{I / H}^{D / H} \psi_n$ is given by 
$$
(\Ind_{I / H}^{D / H} \psi_n) (g) = \begin{cases} 0, & \text{if }g \not\in I / H, \\ \sum_{t \in T} \psi_n (t g t^{-1}), & \text{if } g \in I / H, \end{cases}
$$
where $T$ is a transversal for $(D / H )/ (I / H)$, which we choose to be $T = \langle \sigma \rangle$. For $\tau$ a generator the cyclic group $I / H$ with $\psi_n(\tau) = \zeta_n$, one has $\bQ(\Ind_I^D \psi_n) = \bQ(\Ind_{I / H}^{D / H} \psi_n (\tau)) = \bQ(\sum_{i = 0}^{2^k - 1} \zeta_n^{q^i}) \subset \bQ(\zeta_n)^{\langle \zeta_n \mapsto \zeta_n^q \rangle}$ (where $\sigma$ has order $2^k = f_{\cF / \cK}$ and $q^{2^{k}} \equiv 1 \mod n$ as the extension $\cK' / \cK$ is tamely ramified).  

Thus if $K \subset \bQ(\Ind_I^D \psi_n)$, then $K \subset \bQ(\zeta_n)^{\langle \zeta_n \mapsto \zeta_n^{q} \rangle}$ and so the Artin symbol for $q$ in $K / \bQ$, and hence also for $l$, is trivial. It follows that $l$ splits in $K$ and hence is the norm of an ideal of $K$. 
\end{proof}

We first consider $h_{\fe}(\Psi_n)$ when $\fe \nmid n$.  The following table details the $n$ with $\fe \nmid n$ and $h_{\fe}(\Psi_n) \equiv l \mod \bQ^{\times 2}$, where $p$ is an odd prime. 
\begin{table}[H]\renewcommand{\arraystretch}{1.2}
	\caption{Values of $n$ with $h_{\fe}(\Psi_n) \not\in \bQ^{\times 2}$ and $\fe \nmid n$.}
	\label{table1}
	\begin{tabular}{|l | l|}
	\hline
	& $\{ n \mid \fe \nmid n,\ h_{\fe}(\Psi_n) \equiv l \mod \bQ^{\times 2} \}$ \\
	\hline
	$\fe = 2$ & $\{ p^k \mid k \geq 1,\ p \equiv 3 \mod 4 \}$ \\
	\hline
	$\fe = 3$ & $\{ 2^{k + 1},\ 2 p^{k},\ p^{k} \mid k \geq 1,\ p \equiv -1 \mod 3 \}$ \\
	\hline
	$\fe = 4$ & $\{p^k,\ 2(p')^k \mid k \geq 1,\ p \not\equiv -1 \mod 8,\ p' \equiv \pm 3 \mod 8\}$\\
	\hline
	$\fe = 6$ & $\{2^{k + 2},\ 3^{k + 1},\ p^{k},\ 2(p')^{k}\mid k \geq 1, p \equiv -1, -5 \mod 12, p' \equiv \pm 5 \mod 12 \}$\\
	\hline
\end{tabular}
\end{table}

Observe that the only possible quadratic subfields of $\bQ(\zeta_n)$ for the $n$ in the table are $\bQ(i)$, $\bQ(\sqrt{\pm 2})$, $\bQ(\sqrt{p^*})$. The following lemma states that $l$ is a norm from these fields when they are contained in $\bQ(\Ind_I^D \psi_n)$ and so $h_{\fe}(\Psi_n)^{a_n}$, $(h_{\fe}\cdot g_{\fe})(\Psi_n)^{a_n} \in N_{K / \bQ}(K^{\times})$.    

\begin{lemma}\label{lem:good_fields}
	If $K \in \{\bQ(i), \bQ(\sqrt{\pm 2}), \bQ(\sqrt{p^*}) \}$ and $l$ is the norm of an ideal of $K$, then $l \in N_{K / \bQ}(K^{\times})$.
\end{lemma}

\begin{proof}
	Given an abelian extension $K$ of $\bQ$, its genus field is the largest field $L$ contained in the Hilbert class field of $K$ such that $L / \bQ$ is an abelian extension. We will use the property that if a prime $l$ in $\bQ$ has residue degree $1$ in $L$, then it is the norm of a principal ideal of $K$ (this follows from \cite[VI.\S3, Theorem 3.3]{janusz}). 

	The fields $\bQ(i)$, $\bQ(\sqrt{\pm 2})$ , $\bQ(\sqrt{p^*})$ are equal to their genus fields. As $l$ is the norm of an ideal in these fields, it has residue degree $1$ and hence is the norm of a principal ideal. Since $\bQ(i)$, $\bQ(\sqrt{-2})$, and $\bQ(\sqrt{-p})$ for $p \equiv 3 \mod 4$ are imaginary, $l$ is the norm of an element from these fields. For $\bQ(\sqrt{2})$, $\bQ(\sqrt{p})$ with $p\equiv 1 \mod 4$, $-1$ is the norm of an element from these fields and so $l$ is also.      
\end{proof}

We now consider $n$ such that $\fe \mid n$. The following table details the $n$ with $\fe \mid n$ and $h_{\fe}(\Psi_n) \equiv l \mod \bQ^{\times 2}$, where $p$ is an odd prime.   
\begin{table}[H]\renewcommand{\arraystretch}{1.2}
	\caption{Values of $n$ with $h_{\fe}(\Psi_n) \not\in \bQ^{\times 2}$ and $\fe \mid n$.}    	\label{table2}
	\begin{tabular}{|l | l |}
                \hline
                 & $ \{ n \mid \fe \mid n,\ h(\Psi_n) \equiv l \mod \bQ^{\times 2} \}$ \\
                \hline  
                $\fe = 2$ & $\{ 2^2,\ 2p^k \mid k \geq 1,\ p \equiv 3 \mod 4 \}$  \\
                \hline
		$\fe = 3$ & $ \{ 3,\ 3\cdot 2^k,\ 3 p^k \mid k \geq 1,\ p \equiv -1 \mod 3 \}$ \\
                \hline
                $\fe = 4$ & $\{ 2^2,\ 2^3,\ 2^2(p)^k \mid k \geq 1,\ p \equiv -1 \mod 4\}$ \\
                \hline
		$\fe = 6$ & $\{ 2\cdot 3^k,\ 3\cdot 2^k,\ 2\cdot 3\cdot p^k  \mid k \geq 1,\ p \equiv -1, 5 \mod 12\}$ \\
                \hline
            \end{tabular}
\end{table}

\begin{proposition}\label{prop:case2C}
	In Case $2C$, $h_{\fe}(\Psi_n)^{a_n} \in N_{K / \bQ}(K^{\times})$ for all $n$ such that $\fe \mid n$. 
\end{proposition}

\begin{proof}
	We check that for each $n$ in Table \ref{table2}, $l$ is the norm of an element for all quadratic subfields of $\bQ(\zeta_n)^{\langle \zeta_n \mapsto \zeta_n^{q} \rangle}$. We consider each $\fe$ separately. 

	\underline{Case $\fe = 2$.} This follows from Lemma \ref{lem:good_fields}.

	\underline{Case $\fe = 3$.} The possible quadratic subfields that are not covered by Lemma \ref{lem:good_fields} are $\bQ(\sqrt{3})$, $\bQ(\sqrt{6})$, $\bQ(\sqrt{-6})$, $\bQ(\sqrt{-3p^*})$ for $p \equiv -1 \mod 3$. These have genus fields $\bQ(\sqrt{3})$, $\bQ(\sqrt{6})$, $\bQ(\sqrt{-6}, \sqrt{2})$, and $\bQ(\sqrt{3p})$ when $p \equiv 3 \mod 4$, $\bQ(\sqrt{-3}, \sqrt{p^*})$ when $p \equiv 1 \mod 4$. We know that $l$ has residue degree $1$ in the quadratic fields, and the congruence $q \equiv 1 \mod 3$ ensures that $l$ has residue degree $1$ in each genus field also. Thus $\pm l$ is a norm from each quadratic field, and again using $l \equiv 1 \mod 3$, one can show that $l$ is a norm from each quadratic field by considering the norm equation mod $3$. 

	\underline{Case $\fe = 4$.} We just need to check whether $l$ is a norm from $\bQ(\sqrt{p})$ for $p \equiv -1 \mod 4$. This is its own genus field and so $\pm l$ is the norm of an element from $\bQ(\sqrt{p})$. The congruence $l \equiv 1 \mod 4$ then ensures that $l$ is a norm.

	\underline{Case $\fe = 6$.} This follows from the arguments for $\fe = 3$. 
\end{proof}

\begin{corollary}
	In Case $2C$, $\omega \sim_K d$.  
\end{corollary}

We now turn to Case $2D$.
When $\fe \nmid n$, $g_{\fe}$ is the function on $C_r$ sending a subgroup to its index, and so $g_{\fe}(\Psi_n)^{a_n} = (C_{r / k} \mapsto k)(\Psi_n)^{a_n} \in N_{K / \bQ}(K^{\times})$ by Lemma \ref{lem:f_on_cyclic}.
\begin{proposition}
	In Case $2D$, $(h_{\fe} \cdot g_{\fe})(\Psi_n)^{a_n} \in N_{K / \bQ}(K^{\times})$ for all $n$ such that $\fe \mid n$. Therefore in Case $2D$, $\omega \sim_K d$.  
\end{proposition}

\begin{proof}
	As in the previous proof, to show $\omega \sim_K d$ it suffices to show that $(h_{\fe} \cdot g_{\fe})(\Psi_n)$ is a norm from all quadratic subfields of $L_n := \bQ(\zeta_n)^{\langle \zeta_n \mapsto \zeta_n^{q} \rangle}$ when $n$ is such that $(h_{\fe} \cdot g_{\fe})(\Psi_n) \not\in \bQ^{\times 2}$. This is true for $\fe \nmid n$ by previous results. It is not hard to check that $g_{\fe}(\Psi_n) \not\in \bQ^{\times 2}$ and $\fe \mid n$ implies that $n / \fe$ has at most one prime divisor.  

	The congruence $q \equiv -1 \mod \fe$ restricts the possible quadratic subfields of $L_n$. The result is summarised in the following table for each $\fe \in \{ 3, 4, 6\}$. The second column gives the $n$ such that $(h_{\fe} \cdot g_{\fe})(\Psi_n) \not\in \bQ^{\times 2}$. Depending on some congruence conditions given in the third column, the fourth column gives the quadratic subfields of $L_n$. Arguing similarly to Proposition \ref{prop:case2C}, one can determine that $(h_{\fe} \cdot g_{\fe})(\Psi_n)$ is a norm from the quadratic subfields of $L_n$ in each case. Note in each table that $k \geq 1$, and $p$ is an odd prime unless stated otherwise.
	\begin{table}[H]
		\caption{Non-square values of $(h_{\fe} \cdot g_{\fe})(\Psi_n)$ for $\fe \mid n$,}
		\begin{tabular}{|l |l | l | l | l|}
			\hline
			$\fe$ & $n$ & Conditions & $K \subset L_n$ & $(h_{\fe} \cdot g_{\fe})(\Psi_n)$ \\ \hline
			$ \fe = 4$ & & $(q \mod p,\ p \mod 4)$ &  & \\
			\hline
				   &	$n = 4p^k$ & $(1,\pm 1)$ & $\bQ(\sqrt{p^*})$ & $p$\\ 
				   & & $(-1,1)$ & $\bQ(\sqrt{p})$ & $p$\\
				   & & $(-1, -1)$ & $\bQ(\sqrt{p})$ & $pl$\\ 
				   \hline
			$\fe =  3$ or $6$ & & $q \mod 8$ &  & \\ \hline
				   & $n = 12$ & $1, 5$ & $\bQ(i)$ & $2l$ \\
				   & & $-1,-5$ & $\bQ(\sqrt{3})$ & $2l$ \\
				   \hline
				   & $n = 12\cdot 2^k$ & $1$ & $\bQ(i)$, $\bQ(\sqrt{\pm 2})$ & $2l$\\ 
				   & & $5$ & $\bQ(i)$, $\bQ(\sqrt{\pm 6})$  & $2l$ \\
				   & & $-1$ & $\bQ(\sqrt{2})$, $\bQ(\sqrt{3})$, $\bQ(\sqrt{6})$ & $2l$\\
				   & & $-5$ & $\bQ(\sqrt{-2})$, $\bQ(\sqrt{-6})$, $\bQ(\sqrt{3})$ & $2l$ \\   

		\hline
				   &  & $(q \mod p,\ p \mod 12)$ &  &  \\
		\hline
				   & $n = \fe \cdot p^k,$ $p \not=3$ & $(1, *)$ & $\bQ(\sqrt{p^*})$ & $p$ or $pl$ \\
				   & & $(-1, 1$ or $5)$ & $\bQ(\sqrt{p})$ & $p$ or $pl$ \\
				   &	& $(-1, -1)$ & $\bQ(\sqrt{3p})$ & $pl$ \\
				   &	& $(-1, -5)$ & $\bQ(\sqrt{3p})$ & $p$\\  
		\hline 
			$\fe = 3$ 
				   & $n = 3, 6$ & & none &  $l$ or $2l$\\
		\hline
				 $\fe = 6$ & $n = 6$ &  & none &  $6l$\\
		\hline
				  & $n = 2 \cdot 3^k$ &  & none & $3l$\\ \hline
	\end{tabular}
\end{table}

\end{proof}


\begin{thebibliography}{99}

\bibitem{bartel-selmer}
A. Bartel, Large Selmer groups over number fields, Math. Proc. Cambridge Philos.
Soc. 148 (2010), 73–86.

\bibitem{bartel-dihedral}
A. Bartel, On Brauer-Kuroda type relations of S-class numbers in dihedral extensions,
J. reine angew. Math. 668 (2012), 211–244.

\bibitem{bartel-smit}
A. Bartel, B. de Smit, Index formulae for integral Galois modules, J. London Math.
Soc. 88 (2013), 845–859.

\bibitem{Brauer}
A. Bartel, T. Dokchitser, Brauer relations in finite groups. J. Eur. Math. Soc. 17 (2015), no. 10, pp. 2473–2512

\bibitem{Bartel}
A. Bartel and T. Dokchitser, Rational representations and permutation representations of finite groups, Math. Ann. {\bf 364} (2016), no.~1-2, 539--558; MR3451397

\bibitem{german-brauer}
R. Brauer, Beziehungen zwischen Klassenzahlen von Teilkorpern eines galoisschen
K¨orpers, Math. Nachr. 4 (1951), 158–174.

\bibitem{lbd}
L. Cowland~Kellock and V. Dokchitser, Root numbers and parity phenomena, Bull. Lond. Math. Soc. {\bf 55} (2023), no.~6, 2557--2597; MR4689540

\bibitem{smit}
B. de Smit, Brauer-Kuroda relations for S–class numbers, Acta Arith. 98 no. 2 (2001),
133–146.

\bibitem{annals}
T. Dokchitser and V. Dokchitser, On the Birch-Swinnerton-Dyer quotients modulo squares, Ann. of Math. (2) {\bf 172} (2010), no.~1, 567--596; MR2680426

\bibitem{tamroot} 
T. Dokchitser and V. Dokchitser, Regulator constants and the parity conjecture, Invent. Math. {\bf 178} (2009), no.~1, 23--71; MR2534092

\bibitem{sweep}
T. Dokchitser and V. Dokchitser, Root numbers and parity of ranks of elliptic curves, J. Reine Angew. Math. {\bf 658} (2011), 39--64; MR2831512

\bibitem{dok-wier-ev}
V. Dokchitser, R.~H. Evans and H. Wiersema, On a BSD-type formula for $L$-values of Artin twists of elliptic curves, J. Reine Angew. Math. {\bf 773} (2021), 199--230; MR4237970

\bibitem{dgkm}
V. Dokchitser, H. Green, A. Konstantinou, \& A. Morgan. Parity of ranks of Jacobians of curves.  (2024), https://arxiv.org/abs/2211.06357 

\bibitem{frohlich}
A. Fr\"{o}hlich, J. Queyrut, On the functional equation of the Artin L-function for char- acters of real representations, Invent. Math. 20 (1973), 125–138.

\bibitem{isaacs}
I.~M. Isaacs, {\it Character theory of finite groups}, Pure and Applied Mathematics, No. 69, Academic Press, New York-London, 1976; MR0460423

\bibitem{janusz}
  G.J.Janusz, Algebraic Number Fields, Advances in the Mathematical Sciences, American Mathematical Society, 1996, isbn = 9780821804292.

\bibitem{RohG}
D.~E. Rohrlich, Galois theory, elliptic curves, and root numbers, Compositio Math. {\bf 100} (1996), no.~3, 311--349; MR1387669

\bibitem{Segal}
G.~B. Segal, Permutation representations of finite $p$-groups, Quart. J. Math. Oxford Ser. (2) {\bf 23} (1972), 375--381; MR0322041

\bibitem{Serre}
J.-P. Serre, {\it Linear representations of finite groups}, translated from the second French edition by Leonard L. Scott, 
Graduate Texts in Mathematics, Vol. 42, Springer, New York-Heidelberg, 1977; MR0450380

\bibitem{AdvancedTopics}
Joseph H. Silverman, Advanced Topics in the Arithmetic of Elliptic Curves, Graduate Texts in Mathematics, 1994, Springer New York, NY.

\bibitem{blockmatrices}
John R Silvester, Determinants of Block Matrices, Mathematical Gazette, 2000, 84 (501), pp.460-467, 10.2307/3620776, hal-01509379

\end{thebibliography}
\end{document}